\documentclass[11pt]{amsart}
\usepackage{amsmath,amssymb,latexsym, amsmath, mathtools}
\usepackage{mathrsfs}
\usepackage{hyperref}
\usepackage{xcolor}  
\usepackage[all]{xy}
\usepackage{xspace}
\usepackage{url}

\usepackage[normalem]{ulem}
\usepackage[a4paper]{geometry}
\usepackage{paralist}

\usepackage[backend=biber, style=alphabetic, giveninits]{biblatex}
\hypersetup{
    colorlinks=true, 
    linktoc=all,     
    linkcolor=blue,  
}

\usepackage{lipsum}

\makeatletter

\renewcommand\subsubsection{\@secnumfont}{\bfseries}%
\renewcommand\subsubsection{\@startsection{subsubsection}{3}
  \z@{.5\linespacing\@plus.7\linespacing}{-.5em}%
  {\normalfont\bfseries}}
  
\makeatother

\mathtoolsset{showonlyrefs}
\newcommand{\psmb}{\left( \begin{smallmatrix}}
\newcommand{\psme}{ \end{smallmatrix} \right) }
\numberwithin{equation}{section}
\theoremstyle{plain}
\newtheorem{thm}{Theorem}[section]
\newtheorem{lem}[thm]{Lemma}
\newtheorem{prop}[thm]{Proposition}
\newtheorem{cor}[thm]{Corollary}
\newcommand{\thmref}[1]{Theorem~\ref{#1}}
\newcommand{\lemref}[1]{Lemma~\ref{#1}}
\newcommand{\propref}[1]{Proposition~\ref{#1}}
\newcommand{\corref}[1]{Corollary~\ref{#1}}
\newcommand{\rmkref}[1]{Remark~\ref{#1}}

\theoremstyle{definition}
\newtheorem{rmk}[thm]{Remark}

\newtheorem{conj}[thm]{Conjecture}

\newtheorem{hyp}[thm]{Hypothesis}
\newcommand{\hypref}[1]{Hypothesis~\ref{#1}}

\newcommand*{\Q}{\mathbf{Q}}
\newcommand{\z}{\mf Z}

\newcommand{\h}{\mf H}
\newcommand{\hn}{\mf H_n}
\newcommand{\re}{\mrm{Re}}
\newcommand{\im}{\mrm{Im}}
\newcommand{\mbb}{\mathbb}

\newcommand{\mf}{\mathbf}
\newcommand{\q}{\quad}

\newcommand{\n}{\nonumber}
\newcommand{\mc}{\mathcal}

\newcommand{\mrm}{\mathrm}

\newcommand{\R}{\mf{R}}
\newcommand{\sltwo}{\mrm{SL}_2(\mf Z)}
\newcommand{\spn}{\mrm{Sp}_n(\mf Z)}

\newcommand{\GL}{\mrm{GL}}

\newcommand{\mcr}{\mathscr}
\newcommand{\midmid}{\;\middle|\;}

\newcommand{\symn}{\mrm{Sym}_n(\mf Z)}

\newcommand{\sptwo}{\mrm{Sp}_2( \mf Z)}

\newcommand{\sltwor}{\mrm{SL}_2(\mf R)}

\newcommand{\spnr}{\mrm{Sp}_n( \mf R)}
\newcommand{\glnz}{\mrm{GL}_n( \mf Z)}

\newcommand{\symnr}{\mrm{Sym}_n(\mf R)}

\newcommand{\gp}{\Gamma^{(2)}_0(p)}

\newcommand{\fngam}{\mc F^{(n)}_\Gamma}
\newcommand{\fn}{\mc F^{(n)}_1}

\newcommand{\pfp}{\mf F_p}

\newcommand{\bkngam}{\mathbb B_{k,\Gamma}(Z)}

\newcommand{\skngam}{S^{(n)}_k(\Gamma)}
\newcommand{\skngamconj}{S^{(n)}_k(g^{-1} \Gamma g)}
\newcommand{\skgam}{S^n_k(\Gamma)}
\newcommand{\sab}{\mcr S_\Gamma}

\newcommand{\tr}{\mathrm{tr}\,}
\newcommand{\lan}{\langle }
\newcommand{\ran}{\rangle}

\newcommand{\sumn}{\sum \nolimits}

\newcommand*{\QEDB}{\hfill\ensuremath{\square}}
\newcommand{\smat}[4]{\left( \begin{smallmatrix}#1&#2\\#3&#4\end{smallmatrix}\right)}

\newcommand{\pmat}[4]{\begin{pmatrix} #1 & #2 \\ #3 & #4 \\ \end{pmatrix} }

\newcommand*{\norm}[1]{\left\lVert#1\right\rVert}
\newcommand{\pet}[1]{\lan #1, #1 \ran }


\author{Soumya Das}
\address{Department of Mathematics\\ 
Indian Institute of Science\\ 
Bengaluru -- 560012, India.}
\email{soumya@iisc.ac.in}

\date{}

\begin{document}

\title[Siegel cusp forms and their sizes]{$L^\infty$-sizes of the spaces Siegel cusp forms of degree $n$ via Poincar\'e series}

\begin{abstract}
    We prove the conjectures on the ($L^\infty$)-sizes of the spaces of Siegel cusp forms of degree $n$, weight $k$, for any congruence subgroup in the weight aspect as well as for all principal congruence subgroups in the level aspect, in particular. This size is measured by the size of the Bergman kernel of the space. More precisely we show that the  aforementioned size is $\asymp_{n} k^{3n(n+1)/4}$. Our method uses the Fourier expansion of the Bergman kernel, and has wide applicability. We illustrate this by a simple algorithm.
    We also include some of the applications of our method, including individual sup-norms for small weights and the non-vanishing of Poincar\'e series of exponential type in higher degrees.
\end{abstract}

\maketitle

\section{Introduction}
The sup-norm problem for automorphic forms has seen a lot of activity in the recent years. Obtaining estimates of eigenfunctions of the Laplace-Beltrami operator in terms of the Laplace eigenvalues on a Riemannian manifold has a long history, germinating from classical analysis. We mention the case of $X$ being compact and locally symmetric space, for which  Sarnak proved \cite{sarnak1} the generic bound:
\begin{align} \label{sar-bd}
\|\phi\|_\infty = \sup_{x \in X} |\phi(x)| \ll \lambda_\phi^{\tfrac{\dim X - r}{4}}
\end{align}
where $r$ denotes the rank of $X$. Similar results are also known in the non-compact setting when one restricts the domain to compact subsets (cf. \cite{brumley2020large}).
For number theorists, the main concern is to solve the problem for arithmetic quotients i.e. for symmetric spaces of the form $\Gamma \backslash G$ where $G$ is a Lie group and $\Gamma$ rises from arithmetic input, e.g. $\Gamma$ may be a congruence subgroup of some kind. In this arithmetic setting, due to the availability of a nice family of Hecke operators, one can exploit non-trivial relations among them to implement an `amplification method' to improve upon the bounds for an eigenform that one obtains generally over arbitrary manifolds via classical analysis. The first example of this approach, when $X=\Gamma \backslash \h$ with $\Gamma \subseteq \sltwor$ being a co-compact subgroup arising from the unit group of a quaternion algebra over $\Q$, was shown by Iwaniec and Sarnak \cite{iwaniec1995supnorms}. Their result improves upon \eqref{sar-bd}, and shows that
\[  \|\phi\|_\infty \ll_\epsilon \lambda_\phi^{5/24+\epsilon} \] 
for any $\epsilon>0$. Note that here $\dim X=2, r=1$ and so the generic bound had exponent $1/4$.

In a similar theme as mentioned above, this paper will focus on (scalar-valued weight $k$) holomorphic modular forms on the Siegel upper half space of degree $n \ge 1$ which are automorphic under $\spn$, see \cite{klingen1}, and also Section~\ref{prelim} fo precise definitions. These can be embedded into the space of ``weight-$k$'' Siegel-Maa{\ss}forms  with Laplace eigenvalue $\displaystyle \frac{nk}{4}(n-k+1)$ under the weight $k$ Laplace operator, see \cite{kramer-mandal2}  for more details. The corresponding sup-norm problem, as discussed above, thus becomes one on the non-compact quotient $\spn \backslash \h_n$. 
We will briefly discuss some of the known results of the problem in the ``weight-aspect'' in this setting. For the level aspect, or Laplace eigenvalue aspects, we refer the reader to the works \cite{saha2017sup}, \cite{steiner}.

For the case $n=1$, the best possible bound for the sup-norm of a cuspidal, $L^2$-normalised Hecke eigenform $f$ on $\sltwo$ is due to H. Xia \cite{Xia}. put $\displaystyle \| f\|_\infty:= \sup_{z=x+iy \in \sltwo \backslash \h} y^{k/2} |f(z)| $. Then it is shown in \cite{Xia} that:  $k^{1/4-\epsilon} \ll_\epsilon \| f\|_\infty \ll_\epsilon k^{1/4+\epsilon}$ for any $\epsilon>0$. 
This has now been established for congruence subgroups in \cite{steiner} by  using suitable theta kernels.
While  \cite{steiner2015supnorm} shows that for $f$ as described above, but of half integer weight $k\ge5/2$ contained in the Kohnen plus space of level $4$, one has the bound 
$\| f\|_\infty \ll_{\epsilon} k^{3/7+\epsilon}$.

One can also consider the average version of the above problem -- i.e., consider the sup-norm problem over an orthonormal basis of the space in question,-- and here we might expect better and accurate results. Apart from its own interest, such average results are crucial in implementing the ``amplification-method'' to extract a non-trivial bound for a single $f$.
As we will see below, this average version is tractable in the case of elliptic modular forms, but ceases to become so for $n>1$.

\subsection{The sup-norm problem for holomorphic cusp forms in higher degrees}
Whereas there are plenty of results for the arithmetic quotients of $\sltwo$, very few such are available for the higher dimensional spaces. 
We now focus our attention to the particular case of our interest: the Siegel modular group $\spn$ and its subgroups.  For a cofinite group $\Gamma \subset \spn$, there are two distinct sup-norm problems to consider, each of which is in a nascent stage. 

\subsubsection{The $L^2$-size of the space via the Bergman kernel}
Let $\skngam$ denote the space of holomorphic Siegel cusp form on $\Gamma \subset \spn$, which is co-finite, i.e. the volume of $\Gamma \backslash \h_n$ is finite. In this case we want
 to obtain the \textit{correct} $L^\infty$ size (i.e., genuine asymptotics, instead of bounds) of the space $\skngam$ as $k \to \infty$ as defined below.
Let $B_{k}(\Gamma)$ denote an orthonormal basis for the space $\skngam$, and put (our Petersson norms are not volume normalized, see \eqref{pet-def})
\begin{align}
    \bkngam := \sumn_{F \in B_{k}(\Gamma)} \det(Y)^k |F(Z)|^2, \ (Z \in \hn); \q  \sup(\skngam) : = \sup\nolimits_{Z \in \hn} \bkngam .
\end{align}
Its easy to check that the quantity $\det(Y)^k |F(Z)|^2$ is invariant under $\Gamma$, that $\bkngam$ is independent of the choice of the orthonormal basis, and thus $\sup(\skngam)$ depends only on $n,k$ and $\Gamma$. We wish to study its size in terms of these parameters -- and call $\sup(\skngam)$ the ``$L^2$--size'' or simply the ``size'' of the space $\skngam$. The goal of the paper is to prove optimal bounds for this ``size'' of $\skngam$.

We expect (see also \cite{das-krishna}) that, 
\begin{conj} \label{Wt-conj}
    As $k \to \infty$,
\begin{align} \label{wt-conj}
    \sup\nolimits_{Z \in \mf \hn} \bkngam \asymp_{n,\Gamma} k^{3n(n+1)/4}.
\end{align}
\end{conj}

We now assume that $\Gamma$ is a congruence subgroup of level $N$. In the level aspect, we expect that,
\begin{conj} \label{Lev-conj}
    With $k$ fixed,  as the level $N$ of $\Gamma \to \infty$,
    \begin{align} \label{lev-conj}
    \sup\nolimits_{Z \in \mf \hn} \bkngam \asymp_{n,k} 1.
\end{align}
\end{conj}

It is of course nice if one can obtain a hybrid (i.e. uniform in various of these aspects (weight, degree, level)) result (cf. \cite{kr1}). The genesis of the above conjectures for us are the corresponding lower bounds, which can be obtained without much difficulty, see Section~\ref{lbd-sec} and the known result for the $n=1$ case from \cite{Xia}, \cite{friedman2019effective}. In the ``compact'' setting (i.e., either in the co-compact case or the restriction  to a compact domain) one can essentially work out the expected ``sizes'' by computing the contribution of the term corresponding to the identity in the geometric side of the Bergman kernel (cf. the work of Cogdell-Luo \cite{cog-luo}). This is reminiscent of the computation of the dimension of spaces from trace formulae, see e.g., the remarkable work of \cite{wakatsuki}.
In the non-compact setting, this is not the case, but perhaps the ``parabolic'' contribution gives the answer. It does not seem that these conjectures (which concern arithmetic subgroups of $\spn$) are stated anywhere in the literature pertaining to the complex geometry of the Siegel modular variety. 
But we want to emphasize that in the average situation one would really expect the exact size in terms of the parameters and not merely lower or upper bounds. Perhaps an asymptotic formula holds with appropriate error terms in the above conjectures.

\subsubsection{Sup-norm bounds for an eigenform} 
To obtain bounds going beyond what is implied by the above inequalities for a single newform $F$, and ultimately reaching the limit (inspired by average lower bounds or dimension considerations) we conjecture that in the weight aspect (cf. \cite{blo} for level one Ikeda lifts):
\begin{conj} \label{ind-wt-conj}
 If $F$ is an eigenfunction of all Hecke operators,  as $k \to \infty$,
    \begin{align} \label{ind-wt}
    \norm{F}_\infty := \sup\nolimits_{Z \in \mf \hn} \det(Y)^{k/2}|F(Z)| \ll k^{\frac{n(n+1)}{8}}.
\end{align}
\end{conj}

Similarly in the $N$ aspect (here $\Gamma=\Gamma_0^{(n)}(N)$) we conjecture  that,
\begin{conj} \label{ind-lev-conj}
  With $F, \Gamma$ as above,  as $N \to \infty$,
\begin{align} \label{ind-lev}
    \norm{F}_\infty \ll_{k} N^{-\frac{n(n+1)}{4} }.
\end{align}
\end{conj}
Some comments are  in order on how the above conjectures fit together. We only illustrate this for level $1$. Namely, squaring and adding \eqref{ind-lev} over an orthonormal Hecke-basis, one gets 
\[ \sup\nolimits_{Z \in \mf \hn} \bkngam \ll k^{\frac{n(n+1)}{4}} \cdot \dim S^n_k \ll k^{\frac{n(n+1)}{4}} \cdot k^{\frac{n(n+1)}{2}}=k^{3n(n+1)/4}, \]
since $\dim S^n_k = O(k^{\frac{n(n+1)}{2}})$ with the implied constant depending only on $n$. See \cite[Chapter~2, Theorem~4.5]{andrianov2} for level $1$ and \cite{wakatsuki}  for an asymptotic formula in the case of principal congruence subgroups. Essentially \cite{wakatsuki} says that the dimension is proportional to the index of the discrete subgroup -- this is what we have used in the case of $\Gamma_0^{(n)}(N)$ as well -- using the index formula from \cite[(1)]{klingen2}. With this in mind, same argument as above applies to the level case as well.

\subsection{Some known results in degree \texorpdfstring{$n \ge 2$}{n} }
Among the few prior results on this topic, one had the papers \cite{das-krishna}, where the $n=2$ weight aspect case, i.e., Conjecture~\ref{wt-conj} was proved, along with reasonable bounds for higher degrees. The idea was to use the Fourier expansion ``high in the cusp'', and the geometric side of the Bergman kernel near the boundary of the fundamental domain. However, we note here that in order to handle the Fourier expansion, one needs ``uniform in $k$'' bounds for the Fourier coefficients $a_F(T)$ of a cusp form $F$. Nothing much is known about this when $n>2$, and thus the idea in \cite{das-krishna} was to relate the question to bounding the $L^2$-norm of these Fourier coefficients over an orthonormal basis, which leads to Fourier coefficients of Siegel Poincar\'e series.
However these objects are only tractable when $n=2$, which helped in settling this case. When $n>3$ certain manoeuvres were required, essentially borrowing the bounds from the geometric side of the Bergman kernel even to handle the Fourier expansion, leading to a loss in the sharpness of the results.

Hybrid {\it upper bounds} corresponding to Conjectures~\ref{ind-wt-conj},~\ref{ind-lev-conj} above have recently been announced in the pre-print \cite{kr1} using heat-kernel techniques. Here the idea is to embed the Siegel cusp forms isometrically into the space of Siegel-Maa{\ss}forms of weight $k$, and realizing the Bergman kernel as a limit of the corresponding heat-kernel. Bounds for the heat-kernel then provides bounds for the Bergman kernel.
\footnote{The author however does not understand the bound in eqn.~(4.15) and the third display on p.~52 in \cite{kr1}, where the constant $c_1$ is claimed to be independent of $k,\ell$ or $f$ and has been subsequently used at many further instances, whereas from the quoted reference \cite[p.~57, \S~5]{klingen1}, it is rather clear that this is not the case. Secondly, the authors invoke Theorem~4.3 loc. cit., which is for compact quotients $M=\Gamma \backslash \hn$, with the implied constant depending on $M$ -- to the compact regions seemingly depending on the weight $k$. This seems to interfere with the bounds.} 

Prior to these, one also had the papers \cite{blo} and \cite{blo-pohl}, where the cases of Ikeda lifts (under GLH) and restriction to compacta when $n=2$ (for Siegel-Maa{\ss} forms) were considered respectively.

The aim of this paper is to provide proofs (and an algorithm to implement our idea more generally: see Section~\ref{algo-gam-lat}), of sufficient interest to analytic number theorists, for the above mentioned conjectures on average  using arguments from the perspective of analytic  theory of automorphic forms, viz., essentially using the Fourier expansion of the Bergman kernel of the space. In particular our arguments use neither the heat kernel method (cf. \cite{kr1}), nor the complicated analysis of the Bergman kernel (cf. \cite{das-krishna}, case $n=2$). In this sense they are much simpler. Moreover, it is essentially based on the philosophy already laid down for instance in \cite{das-krishna} -- that the relevant Poincar\'e series or its Fourier coefficients should play a decisive role in controlling the size of the Bergman kernel. In that sense, the present paper may be thought of as a sequel to \cite{das-krishna}.

\subsection{Main results}
We now state the main results of the article. In line with the Conjecture~\ref{ind-wt} above, this essentially confirms \eqref{ind-wt} (up to $\epsilon$).
\begin{thm}[Weight aspect] \label{mainthm-wt}
    Let $k \ge 2n+2$ and $\Gamma  \subset \spn $ be any congruence subgroup of level $N$, for any $n \ge 1$. Let $\epsilon>0$. Then one has 
    \begin{equation}
     \mf c_\Gamma  \, k^{\frac{3n(n+1)}{4}} \ll_n  \sup(\skngam)    \ll_{n, \epsilon} k^{\frac{3n(n+1)}{4}} \, N^{\frac{n^2}{2}+\epsilon}, 
     \end{equation}
where $\mf c_\Gamma$ is a constant depending only on $\Gamma$, and moreover $\mf c_\Gamma \gg_n 1$ for most familiar congruence subgroups -- e.g. for $\Gamma=\Gamma^{(n)}_0(N), \Gamma^{(n)}_1(N), \Gamma^{(n)}(N), \Gamma^{(n),0}(N)$ or for any $\Gamma$ whose ``maximal width at the cusp $\infty$'' is $1$; or
when $N \asymp_n 1$. 
\end{thm}
For the definition of the ``maximal width at the cusp $\infty$'', see \eqref{nij-omega}.

We note that \thmref{mainthm-wt} should hold for any `arithmetic' subgroup of $\spnr$, see Remark~\ref{k-arith}. We however stick to congruence subgroups, as they are of primary interest, and no further new ideas are necessary to handle the arithmetic groups. The upper bound in \thmref{mainthm-wt} is proven in Section~\ref{upp-wt-sec}, and the lower bounds in Section~\ref{lbds-all}.

Coming to the condition on the weight in \thmref{mainthm-wt}, note that in the classical theory of Siegel modular forms, see e.g. \cite{klingen1, freitag}, one usually encounters `large' weights (commonly $k \ge 2n$) to ensure absolute convergence of various objects -- be it Eisenstein/Poincar\'e series or for the validity of higher dimensional sums and integrals. Of course one can adopt the so-called `Hecke's trick' to treat the smaller weights in some cases; and this in turn would perhaps improve the conditions on weights in the results of this paper. In the later parts of this paper we include a different approach on the `small weights', see Subsection~\ref{sm-wt}.
Also, for using comparison of `sum versus integral' (cf. \cite{braun}) in higher dimensions one has to be careful in this respect.

Next, in \thmref{mainthm-prin} given below, we prove a hybrid version of the Conjectures~(i) above, for all the principal congruence subgroups.
The reader may notice that while \thmref{mainthm-wt} is weaker in the level aspect than that of \thmref{mainthm-prin}, it covers all familiar congruence subgroups, and is readily applicable to a wide variety of situations.
Indeed, the proof of \thmref{mainthm-wt} is simpler than that of \thmref{mainthm-prin}, uses a `soft' approach mentioned above, and probably will find use elsewhere. In addition, some of the key calculations in the proof of \thmref{mainthm-wt} are crucially used in Section~\ref{poin-sec}.

\begin{thm}[Hybrid aspect] \label{mainthm-prin}
Let $\Gamma^{(n)}(N) \subset \spn$ be the principal congruence subgroup of level $N$. Then for all $k>2n+2$, one has the uniform bound
\begin{equation}
 \sup(\skngam)  \asymp_n k^{\frac{3n(n+1)}{4}} .
\end{equation}
\end{thm}
We prove a slightly stronger result, see \thmref{mainthm-rest}, from which \thmref{mainthm-prin} follows immediately. The advantage with principal congruence subgroups is that there is only one conjugate (of itself) and the congruence $c_g \equiv 0 \bmod N$ holds for all elements of the group. This can be leveraged suitably to kill the level dependence in \thmref{mainthm-wt}. Probably a more conceptual argument, still within the current framework, can be found. Our calculations simplify if we assume that the group of translations look like $\omega \cdot \symn$ at some cusp. This may not hold for all congruence subgroups when $n \ge 2$, see \rmkref{cusp-config}.

There are two main ideas behind the success of this method over that in \cite{das-krishna}. As discussed above, the first one is the use of Poincar\'e series for all the conjugate subgroups of $\Gamma$ -- this allows one to work only in the convenient level one fundamental domain $\fn$ and completely dispenses with the need to work with the  Petersson trace formula, something which is rather hard to handle anyway in higher degrees, see the introduction of \cite{das-krishna}.
 The second one roughly stems from the fact that Poincar\'e series $P_{G_1}$ for a group $G_1$ is in some sense `smaller' than one for a bigger group $G_2 \supseteq G_1$ as the number of summands is higher for $G_2$. This feature is utilized actually for majorants of these objects, thereby reducing the bulk of the calculations to $\spn$ itself. 

When the weight is  large enough, of the order of the level at a $\log$ scale (i.e., $\log k \gg \log N$), we can prove the desired bound for all congruence subgroups.
\begin{thm}[Hybrid aspect] \label{mainthm-klarge}
Let $\Gamma \subset \spn$ be any congruence subgroup of level $N$. Let $\epsilon>0$. Then for all $k > \max\{(n+1)^2, N^{3+\epsilon} \}$, one has the uniform bound
\begin{equation}
  \mf c_\Gamma \, k^{\frac{3n(n+1)}{4}} \ll_n  \sup(\skngam)  \ll_n k^{\frac{3n(n+1)}{4}} .
\end{equation}
\end{thm}
Note that the constraint here is of a different kind than the above Theorems: the weight has to be large compared to the level, and not only to the degree.
The proof follows by piecing together various bounds on the Bergman kernel that one obtains while proving \thmref{mainthm-wt} and depends on bounding various sums over symmetric matrices in a bit different manner than that in \thmref{mainthm-wt}. See Subsection~\ref{klarge-sec}.

\subsection{Co-compact case}
When the fundamental domain $\Gamma \backslash \hn$ is compact and $k \ge 2n$, one has the bound
\begin{align} \label{cocmpt}
    \sup(\skgam) \ll_n k^{\frac{n(n+1)}{2}} .
\end{align}
This bound is expected from complex geometry, see \cite{kr1} and the references therein.
This is in spirit worked out in \cite{cog-luo}, even though they only consider compact regions in the co-finite setting. One can however get this bound by observing \cite[eqn.~(8.2)]{das-krishna}. Since \cite[eqn.~(8.3)]{das-krishna} (Godement's theorem) notes that the absolute majorant of the Bergman kernel for $\spn$ converges and gives the bound in \eqref{cocmpt}, all co-compact subgroups $\Gamma$ produce the same bound. Indeed, in \cite[p.~81, paragraph~2, last line]{klingen1} it is noted that Godement's theorem is valid for all discontinuous subgroups of $\spn$. Hence in the rest of the paper, we focus only on the case of congruence subgroups.

\subsection{Remarks about further generalizations and applicability}
\subsubsection{Vector-valued cusp forms}
It is clear that the method of this paper generalizes without much issue to vector-valued Siegel cusp forms. We have not included it solely because it is not clear what is the explicit relationship of some of the constants e.g. $c(\rho), H(\rho)$ (for the spaces $S_\rho(\Gamma)$, cf. \cite{godement}) that arise via the reproducing formula of the Bergman kernel or the Gamma function for $\rho$ with the spectral parameters. We could not find a suitable reference for this. It seems that the Cartan Seminars\footnote{\url{https://www.numdam.org/actas/SHC/}} provide the only accessible references, where Godement comments that when $n=2$, these constants may be calculated explicitly, but probably not in higher degrees. Thus one needs to understand this and formulate a suitable conjecture for $S_\rho(\Gamma)$.
This being granted, all one needs, is to bound the relevant Poincar\'e series (say in level one, vector-valued) in Siegel's fundamental domain as:
\begin{align} \label{crucial}
    P_{T, \Gamma}(Z) \ll_n 1 \q \q (Z \in \fn),
\end{align}
in a suitable norm of the underlying vector space. This is no different than our present situation!
After this, it is rather easy to obtain the desired bound (cf. Subsection~\ref{gen-str}~(3)) -- and this is indeed our overall strategy in this paper. 

\subsubsection{Half-integral weights, Hilbert cusp forms etc.}
However the treatment of half-integral weights is verbatim the same, one could include them at the cost of more standard technicalities and notation. We choose not to do this, and leave it to the interested reader to work this out. In particular, we note that the present arguments do not rely on any kind of Atkin-Lehner-Hecke theory or specific arithmetic decomposition of congruence subgroups (cf. \cite[Appendix]{das-anamby2}) to facilitate bounds in convenient regions.  For arbitrary congruence subgroups such information may not be available at present anyway. However the arithmeticity of the subgroups that we consider is important.
The method generalizes to Hilbert modular forms with some work, and also to automorphic forms on other tube domains in principle. Hermitian or orthogonal modular forms are definitely cases of interest here, and it would be worthwhile to try these cases.

\subsection{Lower bounds} 
The other point that we want to emphasize, is about the lower bounds. We believe that the lower bounds are crucial in the problem on average -- as this not only implies the \textit{correct} size of the relevant space; but also aids in setting up the conjectural upper bounds for various spaces where the dimension formulae are not available. Indeed this was the way in which the conjectures for $\spn$ and its subgroups were obtained in \cite{das-krishna}. Of course in level one, an asymptotic dimension formula is known, see e.g. \cite{wakatsuki}. As far as the author is aware, such formulae is not known for arbitrary congruence subgroups in a usable form.
The respective lower bounds were easy to establish for $\spn$ (cf. \cite[Section~7.1]{das-krishna}) -- the presence of levels create some trouble, since the fundamental domain is more complicated -- and delicate arguments are required. In many interesting cases, such lower bounds can be obtained, see Section~\ref{lbds-all}. These rely on a result about the ``soft'' asymptotic formula for the Fourier coefficients of Siegel Poincar\'e series of higher degrees, see Theorem~\ref{pt-asymp}. Since the explicit formula for such Fourier coefficients is not available for higher degrees, our method in this paper plays a vital role in the proof.

\subsection{Small weights} \label{sm-wt}
We now make some remarks about the so-called small weights, say when $k \approx n$ and the level is one. Here the dimension of $S^n_k$ is just a function of $n$; and also since the implied constants in \thmref{mainthm-wt}, \thmref{mainthm-prin} depend on $n$, there is not much to do in the weight aspect. For other congruence subgroups, only the level aspect is relevant and of course, more subtle. In \thmref{small-wt-thm} we show that for any $L^2$-normalised $G \in \skgam$ for $k \ge \frac{n}{2}$ satisfies the bound:
\begin{equation}
    \norm{G}_\infty \ll_n k^{\frac{3n(n+1)}{8}}
\end{equation}
where we assume the bound $\sup(S^n_\kappa(\Gamma)) \ll_n \kappa^{\frac{3n(n+1)}{4}}$ which the expected  from \eqref{wt-conj} and \eqref{lev-conj}, on \textit{any} congruence subgroup $\Gamma$ for {\it large enough} weights $\kappa$ -- by a norm interpolation technique using H\"older's inequality. The condition on the weight is the best possible, as anything lower would not admit cusp forms. Again this method is quite general.

\subsection{Applications of our method}
As immediate applications of our method, we can shed some light on the rather hard conjecture on the non-vanishing of Siegel Poincar\'e series (of exponential type). We essentially can show that given any set of pairwise $\glnz$ inequivalent matrices $\{T_1, \ldots, T_h \} \subset \Lambda_n$, the corresponding Poincar\'e series $\{ P_{T_1}, \ldots, P_{T_h} \}$ on $\Gamma^{(n)}_0(N)$ are linearly independent provided $k$ is large enough, depending on the traces of these matrices. This generalizes and improves upon the main results of \cite{boech-das} where the same statement was proved for infinite families of weights $k$ in arithmetic progressions modulo suitable primes, for $\Gamma=\Gamma^{(n)}_0(N)$ via congruences between modular forms. See \corref{poin-cor} and \corref{poin-nonvanish}. Our results also generalize (and reprove when $n=2$) some of the results in \cite{zhining} to higher degrees.

We also work out an uniform bound for the Bergman kernel $\bkngam$ in the Siegel's fundamental domain for $\Gamma$, see \corref{pol-bd-fn}.

There are are further applications of our method to the problem of mass equidistribution on average on the Siegel modular variety, as well as to the bounds on certain $L$-series arising from cusp forms -- these will be treated in a forthcoming paper. It is however not clear what our method implies for the various subspaces of \textit{lifts}, e.g. see \cite{das-anamby1}, \cite{das-anamby2} for Saito-Kurokawa lifts. It would be interesting to investigate this further.

\subsection{Appendix: elliptic modular forms}
We end the introduction by mentioning the Appendix in Section~\ref{app}, where we treat the simpler case of elliptic modular forms. When the level $N$ of $\Gamma_0(N)$ is square-free then it is known from \cite{jorgenson2004bounding} that the size of $S_2(\Gamma_0(N))$ is $O(1)$.
When $k>2$, our argument is just a few lines, whereas when $k=2$, we first prove a large sieve inequality for \textit{arbitrary} congruence subgroups and apply it to prove the bound $O(\log N)$ for the size of $S_2(\Gamma)$, where the level of $\Gamma$ is $N$. We could not find this result in the literature. Moreover treating all congruence subgroups is crucial for our approach to work, as  we work with all the conjugates of $\Gamma$. See \thmref{duke-gam}~(A),~(B). It seems that the calculations in \cite{das-anamby2}, adapted for $k=2$, might be able to show the bound $O(1)$ (using the conditional summation of Poincar\'e series), but we could not be successful in this regard. 

We give two methods for proving the large sieve inequality.
The first one is Duke-Iwaniec's method \cite{duke} of covering suitable subsets of $\h$ of finite volume by finitely many copies of the fundamental domain. The second one is due to Duke-Friedlander-Iwaniec \cite{dfi} via the Petersson formula -- here we feel that the situation is rather subtle, since for an arbitrary congruence subgroup one does not have complete control over all the arithmetic functions at hand. For example, we have to assume a certain uniform bound for the Kloosterman sums $S_\Gamma(m,n;c)$, see \hypref{kloo-hypo}. Its likely that a suitable variant of our results also hold when $k=1$. Also one might try to consider the large sieve machinery for higher degrees -- for arbitrary congruence subgroups this seems to be difficult. For $\Gamma^{(n)}_0(N)$, there may be some hope, and this is under consideration by the author.

\subsection*{Acknowledgements}
{\small
It is a pleasure to thank Pramath Anamby for comments, and the anonymous referee for many insightful comments and suggestions that improved the paper.
S.D. thanks IISc. Bangalore, UGC Centre for Advanced Studies, DST India for financial support.
}

\section{Notation and setting} \label{prelim}

\subsection{Siegel modular forms}  \label{smf-def}
For basic facts about Siegel modular forms we refer to \cite{andrianov2}, \cite{freitag} or 
\cite{klingen1}. The symplectic group $\spnr$ defined as
\[ \spnr = \{ g= \smat{A}{B}{C}{D} \in M({2n},\R) \mid M^t \smat{0_n}{-1_n}{1_n}{0_n}M= \smat{0_n}{-1_n}{1_n}{0_n} \} ,\]
acts on Siegel's half-space $\mf H_n$ defined as 
 \[  { \mf H_n := \{ Z= Z^t \in M(n, \mf C) \mid (Z - \overline{Z})/2i \text{ is positive definite} \}, } \] 
in the usual way by $g \langle Z \rangle=(AZ+B)(CZ+D)^{-1}$ and
on functions $F:{\mf H}_n\longrightarrow {\mf C}$ by the {weight $k$} \lq\lq stroke\rq\rq operator:
\[ (F\mid_kg)(Z)= \det(CZ+D)^{-k} F(g \langle Z \rangle) \qquad  (g=\left(
\begin{smallmatrix}
A & B\\ C & D\end{smallmatrix}\right)\in \mrm{Sp}(n, \mf R), k \in \mf Z_{\geq0} ). \] 
Throughout the paper we put $\Gamma_n=\spn=\spnr \cap M(2n, \z)$.
For $g \in \Gamma_n$, we also write $g= \psmb a_g & b_g \\ c_g & d_g \psme$. We write $\displaystyle J(g,Z) := \det(c_g Z+d_g)$ as the automorphy factor. We also put $\h:= \mf H_1$.

Let $\Gamma$ be a congruence subgroup of $\Gamma_n$. For a congruence subgroup $\Gamma$, the level $N$ is defined to be the smallest positive integer such that $\Gamma^{(n)}(N) \subset \Gamma$. Among specific subgroups of $\Gamma_n$, we will mainly consider the following congruence subgroups of $\Gamma_n$.
\begin{align}
    \Gamma_0^{(n)}(N)&:=\{\smat{A}{B}{C}{D}\in \Gamma_n: C\equiv 0\bmod N\};\\
     \Gamma^{(n),0}(N)&:=\{\smat{A}{B}{C}{D}\in \Gamma_n : B\equiv 0\bmod N\};\\
     \Gamma^{(n),0}_0(N)&:=  \Gamma_0^{(n)}(N)  \cap \Gamma^{(n),0}(N) ;\\
\Gamma_1^{(n)}(N)&:= \{\smat{A}{B}{C}{D}\in \Gamma_n: C\equiv 0\bmod N, A\equiv 1_n\bmod N \} \\
\Gamma^{(n)}(N)&:= \{\smat{A}{B}{C}{D}\in \Gamma_n: B,C\equiv 0\bmod N, A,D \equiv 1_n\bmod N  \}.
\end{align}

A holomorphic function $F$ on
${\mf H}_n$ is called a modular form for $\Gamma$ of weight $k$ if it satisfies the transformation law
\[ F\mid_k\gamma= F \qquad (\gamma=\left(\begin{smallmatrix}
A & B \\ C & D\end{smallmatrix}\right)\in \Gamma), \]
with the additional condition of being holomorphic at the cusps when $n=1$.
We denote the space of all such functions by $M^n_k(\Gamma)$. 

\subsection{Translations, Fourier expansion and cusp forms}
Let $\symn$ denote the group of all $n\times n$ symmetric matrices over $\z$.
We denote by 
\begin{equation}
\Gamma_\infty = \{ \pm \psmb 1_n &  S \\ 0_n & 1_n \psme \mid S=S^t \in \mcr S_\Gamma \} ; \q (\mcr S_\Gamma \subset \symn \text{ is a subgroup})
\end{equation}
the subgroup consisting of the translations in $\displaystyle \Gamma$, and $\Gamma_{0,\infty} = \{ g \in \Gamma \mid c_g=0_n\}$. We further let 
\begin{equation}
    \mathscr U(\Gamma)=\{  m(U) \in \Gamma \mid U \in \glnz \} \q (  m(
    U)=\psmb U^{t} & 0\\ 0 & U^{-1} \psme),
\end{equation}
be the subgroup of the embedded copy of $\glnz$ inside $\Gamma$. Further, we define the Siegel parabolic subgroup 
\begin{equation}
\Gamma_{0,\infty} := \{ g \in \Gamma \mid c_g=0 \},
\end{equation}
and the subgroup of $\glnz$ embedded in $\Gamma$ by
\begin{equation} \label{gln-gam}
    \mcr U_\Gamma := \{ U \in \glnz \mid \psmb
        U^t & 0 \\ 0 & U^{-1}
    \psme \in \Gamma  \} \q (=m^{-1}(\mcr U(\Gamma))).
\end{equation}
Note that $\mcr U_{\Gamma_n} = \glnz$. Note that every element $g \in \Gamma_{0,\infty}$ can be uniquely represented, in the above notation, as $g =  m(U)g_1$, where $g_1 \in \Gamma_{0,\infty} \backslash \Gamma$.

The lattice dual (w.r.t. to the trace pairing) to $\mcr S_\Gamma$ will be denoted as $\Lambda^*_\Gamma$, i.e., $\displaystyle \Lambda^*_\Gamma=\{ T=T^t \in M_n(\Q) \mid \tr (TS)\in \z \, \text{for all } S \in \mcr S_\Gamma\}$. It consists of matrices $T$ such that $\displaystyle t_{ij} \in \frac{1}{2 n_{ij}} \z$. For each $(i,j)$, the positive integer $n_{ij}$ is defined to be the generator of the subgroup of $\z$ consisting of the integers $s_{ij}$, as $S=(s_{pq})$ varies in $\mcr S_\Gamma$.
When there is no confusion we drop the dependence on $\Gamma$ from the notation. We denote by $\Lambda_\Gamma:= \{ T \in \Lambda^*_\Gamma \mid T>0 \}$.

Every $F \in M^n_k(\Gamma)$ can be expressed by a Fourier series of the form
\begin{equation} \label{fe}
    F(Z) = \sumn_{T \in \Lambda^*_\Gamma, T \ge 0 } a_F(T) e(TZ);
\end{equation}
and the `denominator' of $T$ is $\omega$, i.e., it is the smallest positive integer such that $T \in \frac{1}{\omega} \Lambda^*_n$ (cf. \eqref{nij-omega}), where $\Lambda_n$ denotes the usual half-integral positive semi-definite matrices of size $n$. It is in this sense $\omega$ may be interpreted as the `maximal-width of the cusp' at $\infty$ of $\Gamma$. We also refer to the $n_{ij}$ as the cusp width at $(i,j)$. See the next section.

For each $g \in \Gamma_n$, the modular form $F|_k g$ on $g^{-1}\Gamma g$ has a Fourier expansion like \eqref{fe}, and if for all such $g$, the Fourier expansion of $F|_k g$ is supported only on $\Lambda_{g^{-1} \Gamma g}$, i.e., on the cone of the positive definite half-integral matrices,  we say that $F$ is a cusp form.
Let $\skgam$ denote the space of holomorphic Siegel cusp forms on $\Gamma$ of scalar weight $k$.

If $F,G \in \skngam$, we define their Petersson inner product, denoted by $\lan F,G \ran=\lan F,G \ran_\Gamma$ by the integral pairing
\begin{align} \label{pet-def}
    \lan F,G \ran = \int_{\Gamma \backslash \hn} F(Z) \overline{G(Z)} \det(Y)^k d \mu(Z),
\end{align}
where $d \mu(Z) = dX \, dY \det(Y)^{-\frac{n+1}{2}}$ is the invariant volume element on $\hn$ and $\Gamma \backslash \hn$ denotes any fundamental domain for the action of $\Gamma$ on $\hn$.

\subsection{Cusp configuration}
Let us discuss a bit more about the cusp at $\infty$. Here there are two quantities which are relevant: the index and the ``maximal width at $\infty$'' (and analogously at other cusps):
\begin{align} \label{nij-omega}
    |\mcr S_\Gamma| := |\symn \colon \mcr S_\Gamma|= \prod\nolimits_{i \le j} n_{ij}(\Gamma), \q \omega=\omega_\Gamma := \mrm{lcm.}_{i \le j}(n_{ij}).
\end{align}
We call the collection $\mc C(\Gamma):= \{ n_{ij}(\Gamma) \}$ to be the cusp-width configuration for $\Gamma$ at the cusp $\infty$.

For the principal congruence subgroup $\Gamma^{(n)}(N)$ (see Section~\ref{smf-def}), things simplify a bit: one has the exact description $\mcr S_\Gamma = N \cdot \symn$, $ |\mcr S_\Gamma| =N^{\frac{n(n+1)}{2}}$ and $\omega_\Gamma=N$. Unfortunately not all congruence subgroups have this simple description as the example below shows, unlike when $n=1$. Clearly $\omega_\Gamma | N$ for all $\Gamma$. Moreover for at least one conjugate $G$ of $\Gamma$, one must have $\omega_G=N$.

\begin{rmk} \label{cusp-config}
    In this remark we want to show that the cusp-width configuration at $\infty$ may not be constant, i.e., $n_{ij}(\Gamma)$ are not constant multiples of a fixed integer.
    We show this for some conjugates of $\Gamma:=\Gamma^{(2)}_0(p)$ as examples.

    We take a coset representative $g = \psmb A & B \\ C & D \psme$ in $\Gamma \backslash \Gamma_n$. We can always  assume without loss  $C$ is non-singular (cf. the proof of Corollary~\ref{pol-bd-fn}), but we won't use this. We just assume $C$ is upper-triangular -- and will choose specific representatives, in this form soon.

Let $G= g^{-1} \Gamma g$.
    For $G$, let $\displaystyle n_{11}=t_1, n_{12}=t_3, n_{22}=t_2$. Then the generators of  the lattice $\mcr S_G$ are $S_1=t_1 \psmb 1 & 0 \\ 0 & 0 \psme$, $S_3=t_3 \psmb 0 & 0 \\ 0 & 1 \psme$ and $S_2=t_2 \psmb 0 & 1 \\ 1 & 0 \psme$ respectively. Then $t_j$ are the smallest positive integers such that
    \begin{align}
        \psmb 1 & S_j \\ 0 & 1 \psme \in G, \text{ that is } g \psmb 1 & S_j \\ 0 & 1 \psme g^{-1} \in \Gamma^{(2)}_0(p).
    \end{align}
A calculation shows that $C S_j C^t \equiv 0 \bmod p$ for any $C$ as above.  Letting $C= \psmb c_1 & c_2 \\ 0 & c_4 \psme$, this means
\begin{align}
    t_1 \psmb c^2_1 & 0 \\ 0 & 0 \psme \equiv 0 \bmod p, \, t_3 \psmb 2 c_1c_2 & c_1c_4 \\ c_1c_4 & 0 \psme \equiv 0 \bmod p, \, t_2 \psmb c^2_2 & c_2c_4 \\ c_2c_4 & c^2_4 \psme \equiv 0 \bmod p.
\end{align}
One can immediately translate these into divisibility conditions
\begin{align} \label{div}
    \frac{p}{t_1} \mid c_1^2, \q  \frac{p}{t_3} \mid (2 c_1c_2, c_1c_4), \q \frac{p}{t_2} \mid (c_2^2,c_4^2,c_2c_4).
\end{align}
    Now we recall the explicit set of coset representatives in question. These can be of three types  (see e. g. \cite{boech-naga}, \cite{das-anamby2}). More precisely, for $0\le j\le 2$ put
\begin{equation} \label{wj}
    w_j=\begin{pmatrix}
        0_j & 0 & -1_{j} & 0\\
        0 & 1_{n-j} & 0 & 0_{n-j}\\
        1_{j} & 0 & 0_{j} & 0\\
        0 & 0_{n-j} & 0 & 1_{n-j}
    \end{pmatrix} = \begin{pmatrix}
        a_j & b_j \\ c_j & d_j
    \end{pmatrix}.
\end{equation}
Further, for $A\in \GL_2(\mf Z)$ and $B\in M_2 (\mf Z)$, we put
$\displaystyle m(A):=\smat{A^t}{0}{0}{A^{-1}} \text{ and } n(B):=\smat{1}{B}{0}{1}$. Note that $m(A)$ of this paper is $m(A^t)$ in \cite{das-anamby2}.

    Let $\mc R(j)$ denote the set of matrices of the form $w_j n(B_j) m(A^t)$, where $B_j$ and $A$ are as in \cite{BNmodp} (lifts of matrices in $\mrm{Sym}_j(\mf F_p)$ and $ \GL_2(\mf F_p) /P_{2,j}$ to $\sptwo$, respectively). Then a set of right coset representatives for $\gp \backslash \sptwo$ is given by
\begin{equation}
    \gp \backslash \sptwo =\bigcup \nolimits_{j=0}^{2} \mc R(j).
\end{equation}
One may take representatives for $ \GL_2(\mf F_p) /P_{2,j}$ to be
\begin{align}
    \pmat{1}{0}{0}{1}, \q \pmat{0}{1}{-1}{x} \, (x \in \pfp).
\end{align}
    
We need to find $C=c_{\gamma}$ for each of the $\gamma \in \cup_j \mc R(j)$, and apply \eqref{div} to them. 
$\mc R(0)$ corresponds to identity, i.e. $\Gamma^{(2)}_0(p)$, so here $C=0$ and all the $t_i=1$. For $\mc R(2)$, one easily checks that $C=Id.$ which means all the $t_i=p$. In the case of $\mc R(1)$, with $c_1$ as in \eqref{wj}, one finds that $C=c_1 A^t \in  \{ \psmb a_1 & a_2 \\ 0 & 0 \psme \}$ with $(a_1,a_2) = (1,0), (0,-1)$. This gives in the first case $t_1=p$ and no conditions on $t_2, t_3$ -- so by minimality $t_2=t_3=1$. Similarly for the remaining case one finds that $t_1=1=t_3, t_2=p$. The the cusp-width configuration at all the cusps look like either $(1,1,1)$ or $(1,1,p)$ or $(p,1,1)$ or $(p,p,p)$.
\end{rmk}

\subsection{Siegel Poincar\'e series}
For a $T \in \Lambda_\Gamma$, we define the Poincar\'e series by
\begin{align} \label{pt-exp}
    P_T(Z)= P_{\Gamma;T}(Z) := \sumn_{\gamma \in \Gamma_\infty \backslash \Gamma} e(T Z) \mid_k \gamma,
\end{align}
which is an element of $\skgam$ provided $k>2n+1$.

For $T \in \Lambda_\Gamma, F \in \skgam$, one has the formula, which is proved by the usual unfolding method (see e.g. \cite[Theorem~8.2.3~(c)]{coh-str} for the case $n=1$)
\begin{align} \label{pt-peter}
    \lan F,  P_{T;G} \ran = \frac{ |\mcr S_\Gamma| \, c_{n,k} \,  a_F(T)}{\det( T)^{k-\frac{n+1}{2}}}, 
\end{align}
where we have put
\begin{align} \label{cnk-def}
   c_{n,k} = \pi^{\frac{n(n-1)}{4}} (4 \pi)^{\frac{n(n+1)}{2}-nk}\Gamma_n \left(k-\frac{n+1}{2} \right).
\end{align}
Let us further define the constants
\begin{align}
    a_{n,k} &:= 2^{-n(n+3)/2}\pi^{-n(n+1)/2}
      \prod \nolimits_{v=1}^{n}{\frac{\Gamma(k-\frac{v-1}{2})}{\Gamma(k-\frac{v+n}{2})}}
      \asymp_n k^{\frac{n(n+1)}{2}}, \\
      b_{n,k} &:= (2 \sqrt{\pi})^{-n(n-1)/2}  (-2\pi i)^{-nk} \prod \nolimits_{\nu =0}^{n-1} \Gamma(k-\nu/2). \label{bnk}
\end{align}
Note that our constant $a_{n,k}$ differs from that in \cite{klingen1} by a factor of $(2i)^{nk}$ due to different normalization. Our setting matches with that of \cite{cog-luo}. Further we record here the relationship between the $L^2$-average of the Fourier coefficients of cusp form in an orthonormal basis and the Fourier coefficients of Siegel Poincar\'e series:
\begin{align} \label{sumsq-fc}
   \sumn_{F \in \mcr B(\Gamma)} \, |a_F(T_0)|^2 =  c_{n,k}^{-1} \, \frac{|\mcr S_\Gamma|^{-1} \, \det(T_0)^{-\frac{n+1}{2}}}{  \Gamma_n(k-\frac{n+1}{2})} \cdot p_{T_0}(T_0).
\end{align}
We remind the reader that \eqref{sumsq-fc} is derived by writing $P_T$ as a linear combination of $F \in \mcr B(\Gamma)$, using orthogonality and \eqref{pt-peter}.

Next, we note the following Lipschitz formula (see e.g. \cite[p.~89]{klingen1}) relative to sub-lattices (for $k>n$)
\begin{align} \label{sym-sum}
    \sumn_{S \in \mcr S_\Gamma} \det(Z+S)^{-k} = \frac{b_{n,k}^{-1}}{|\mcr S_\Gamma|} \sumn_{T \in \Lambda_\Gamma} \det(T)^{k- \frac{n+1}{2}} \, e(TZ),
\end{align}
which comes from the Poisson summation formula along with holomorphy in $Z$ and the K\"ocher's principle if $n \ge 2$. These formulae together leads to the familiar geometric side of the Bergman kernel -- which we mention below, for the seemingly lack of an explicit reference for arbitrary $\Gamma$. The main point is that one needs to be sure about  the level dependence in the formulae explicitly. First we recall that the Bergman kernel is defined by
\begin{align}
    B_{k,\Gamma}(Z,W) := \sumn_{F \in \mcr B_{k,\Gamma}} F(Z) \overline{F(W)},
\end{align}
where $\mcr B_{k,\Gamma}$ denotes an orthonormal basis of $\skgam$. 

Then one can write, using \eqref{pt-peter} and \eqref{sym-sum} successively to get a geometric expression for the Bergman kernel (see e.g. \cite[p.~90, first display]{klingen1}):
\begin{align} \label{bergdef0}
     B_{k,\Gamma}(Z,W) &=  c_{n,k}^{-1} \frac{1}{|\mcr S_\Gamma|} \, \sum_{T \in \Lambda_\Gamma}  \det(T)^{k-\frac{n+1}{2} } P_{T; \Gamma}(Z) e(-T \overline{W} )\\
    &= a_{n,k} b_{n,k}^{-1}\sum_{g \in \Gamma_\infty \backslash \Gamma} J(g,Z)^{-k} \sum_{S \in \mcr S_\Gamma} \det \left( \frac{g(Z) -\overline{W} +S}{2i} \right)^{-k} \\
    & = \frac{a_{n,k}}{2} \sum_{g \in \Gamma} \det \Big(\frac{Z- \overline{W}}{2i}\Big)|^{(n)}_k g.\label{bergdef1}
\end{align}
The extra $2^{-1}$ differs from that in \cite{klingen1} since in \cite{klingen1}, the sum over $g$ is taken $\bmod{\{\pm1 \}}$.
We will use the following notation for the invariant kernel:
\begin{align} \label{bb-bk}
    \mbb B_{k,\Gamma}(Z,W) =   \frac{a_{n,k}}{2} \,  \det(Y)^{k/2} \det(V)^{k/2} \, B_{k, \Gamma}(Z,W).
\end{align}
We  put $B_{k, \Gamma}(Z) :=B_{k, \Gamma}(Z,Z)$, $\mbb B_{k, \Gamma}(Z) :=\mbb B_{k, \Gamma}(Z,Z)$ etc.. 

In view of future applications, we define, for any congruence subgroup $\Gamma \subset \Gamma_n$ and any $\mcr Y>0$, $T \in \Lambda_n$, 
\begin{align} \label{u-gamma}
     H_{\Gamma}(T, \mcr Y) :=   \sumn_{U \in \mcr U_\Gamma} \exp(- 2 \pi \tr \, T \mcr Y[U] ).
    \end{align}
One might think of $H_{\Gamma}(T, \mcr Y)$, say for $\Gamma=\Gamma_n$ as a sub-series of the usual theta series $\displaystyle \vartheta(iV)$ where for $Z \in \hn$, we put $\displaystyle \vartheta(Z) = \sum \nolimits_{X \in \z^{n \times n}} e( Z[X])$. For future usage, we record here the simple property:
\begin{align} \label{hgamma-ineq}
    H_{\Gamma}(T, \mcr Y) \le H_{\Gamma'}(T, \mcr Y),
\end{align}
wherever $\Gamma \subseteq \Gamma'$.

\section{Proof of the  Theorem~\ref{mainthm-wt}: weight aspect -- upper bounds} \label{upp-wt-sec}
We will first prove the statement for the group $\Gamma_n$ and then reduce the case of arbitrary $\Gamma$ to it. This proof is shorter than the hybrid-level aspect, and we believe is worth presenting separately.

\subsection{The case \texorpdfstring{$\Gamma=\Gamma_n$}{}}
We know from \eqref{bb-bk} that
\begin{align} \label{ank}
    \mbb B_{k,\Gamma_n}(Z) =   \frac{a_{n,k}}{2} \, \det(Y)^k B_k(Z).
\end{align}
We appeal to the Fourier expansion of the $B_k(Z,W)$, and put $Z=W$ to get
\begin{align} \label{bk-fe}
    B_k(Z) = b_{n,k}^{-1} \sumn_{T \in \Lambda_n} \det(T)^{k-\frac{n+1}{2} } P_T(Z) e(-T \overline{Z} ). 
\end{align}
First, we note that without loss, one can assume that $Z \in  \fn$ which implies in particular that $Y \gg 1$ and is Minkowski-reduced. 

Next, note that it is enough to show that for $Z \in \mc F^{(n)}_1$ one has
\begin{align} \label{pt-bound}
    P_T(Z) \ll_n 1.
\end{align}

Indeed, if we assume this, then
\begin{align} \label{bkz-fe-bd}
     B_k(Z) \ll b_{n,k}^{-1}  \sumn_T \det(T)^{k-\frac{n+1}{2} } \exp(- 2 \pi \tr T Y).
\end{align}

We know via Lipschitz's formula that for $k>n$ and any $Y_0>0$,
\begin{align} \label{lip-1}
\sumn_{S \in \symn} \det(iY_0+S)^{-k} = b_{n,k}^{-1} \sumn_T \det(T)^{k-\frac{n+1}{2} } \exp(- 2 \pi \tr T Y_0).
\end{align}
Quoting from \cite[Lemma~6.2]{das-krishna}, which is valid for any $Z \in \h_n$, we get, with $Z=iY_0$ that
\begin{align} \label{lip-bd}
    \sumn_S \det(iY_0+S)^{-k} \ll k^{\frac{n(n+1)}{4}} \det(Y_0)^{-k}.
\end{align}
Finally, gathering together \eqref{ank}, \eqref{bkz-fe-bd} and \eqref{lip-bd} we arrive at the desired bound:
\begin{align} \label{wt-pf}
    \mbb B_k(Z) \ll a_{n,k} \cdot k^{\frac{n(n+1)}{4}} \asymp k^{\frac{3n(n+1)}{4}}.
\end{align}

It therefore remains to show \eqref{pt-bound}, i.e., $P_T(Z) \ll 1$. Recall from the definition \eqref{pt-exp} that
\begin{align} \label{pt-exp-def}
    P_T(Z)= \sumn_{g \in \Gamma_\infty \backslash \Gamma_n} J(g,Z)^{-k} e(T g(Z)).
\end{align}
Such a $g \in \Gamma_n$ can be uniquely written as $g=m(U) \gamma$, where $U \in \mcr U_{\Gamma_n}$ and $\gamma \in \Gamma_{0,\infty}$. Then it follows that,
\begin{align} \label{pt-exp-U}
P_T(Z)= \sumn_{\gamma \in (\Gamma_n)_{ {0,\infty}} \backslash \Gamma_n} J(\gamma,Z)^{-k} \sumn_{U}  \det(U)^k e(T U^t \gamma(Z) U).
\end{align}

Let us discuss some heuristics about the function $H_{\Gamma_n}(T, \mcr Y)$ defined in \eqref{u-gamma}.
Via Poisson summation, the theta  series $\vartheta(Z)$ has the inversion formula $\vartheta(Z) = \det(-Z)^{-\frac{n}{2}} \vartheta(-Z^{-1})$ cf. \cite{andrianov2}. If we remove the $X=0$ term, one may expect that $\vartheta(iV)-1 \approx \det(V)^{-n/2}$ if $V$ is `small'. Based on this intution, we have the following proposition.

\begin{prop} \label{hty-lev1}
    For  $\mcr Y>0$, $T \in \Lambda_n$, and for any $D>n/2$, one has
    $\displaystyle H_{\Gamma_n}(T, \mcr Y) \ll_{D,n}  \frac{1}{\det(\mcr Y)^D}$.
\end{prop}

Let us assume \propref{hty-lev1} for the moment and finish the proof of \thmref{mainthm-wt}. Namely, from \eqref{pt-exp-U} we can write
\begin{align} \label{pt-ht}
    P_T(Z) \ll \sumn_{\gamma \in \Gamma_{0,\infty}\backslash \Gamma_n} |J(\gamma,Z)|^{-k} H_{\Gamma_n}(T, Y_\gamma) ,
\end{align}
where $Y_\gamma:= \im \gamma(Z)$. Then from \propref{hty-lev1}, we get the bound
\begin{align} \label{th-bd}
   H_{\Gamma_n}(T, Y_\gamma)   \ll \det(Y_\gamma)^{-n/2-\epsilon} =  \det(Y)^{-n/2-\epsilon} |\det(c_\gamma Z + d_\gamma)|^{n+2\epsilon}.
\end{align}
Using this in \eqref{pt-ht} we get
\begin{align} \label{pt-bd}
    P_T(Z) \ll \sumn_{\gamma \in \Gamma_{0,\infty}\backslash \Gamma_n } |\det(c_\gamma Z + d_\gamma)|^{-k+n+2\epsilon} \ll 1,
\end{align}
for any $Z \in \fn$ and $k \ge 2n+2$, if we invoke argument from \cite[eq.~(6.28) onwards]{das-krishna} and choose $\epsilon$ small depending only on $n$. We already know from from the arguments leading to \eqref{wt-pf} that this implies \thmref{mainthm-wt}. \QEDB

\subsection{Proof of \propref{hty-lev1}}
It remains to prove \propref{hty-lev1}.
Since $ \tr T  \mcr Y[U] = \tr T[U^t] \mcr Y$, by absolute convergence, and because the desired bounds depend only on $\det(\mcr Y)$, clearly we can assume that both $T$ and $\mcr Y$ are reduced. That $T$ can be assumed to be Minkowski-reduced (see \cite[Chapter~I, Section~2]{klingen1}) of course also follows from the definition of $P_T$ in level one. Then
\begin{align} \label{ht}
   H_{\Gamma_n}(T, \mcr Y)  \ll \sumn_{U}  e(- c_n \tr \, \mrm{diag}(T) \mcr Y[U^t]) \ll \sumn_{U}  e(- c_n \tr \mcr Y[U^t]) =  H_{\Gamma_n}(1_n, c_n \mcr Y),
\end{align}
where $\mrm{diag}(T) $ denotes the matrix consisting of the diagonal elements of $T$.

For a matrix $M$, let $M_j$ denote its $j$-th column. Then for some constant $c_n$ depending only on $n$,
\begin{align}
  H_{\Gamma_n}(1_n, c_n \mcr Y)   &\le  \sumn_{U} \exp(- c_n \tr \, U^tU \cdot \mrm{diag}(\mcr Y)) \\
    & =  \sumn_{U} \exp(- c_n  \, \sumn_{j=1}^n v_j \cdot U_j^tU_j  ) 
         \le \prod_{j=1}^n \big( \sumn_{U_j, \mrm{g.c.d.}(U_j)=1} \exp(- c_n  \, \sum_{j=1}^n v_j \cdot U_j^tU_j  ) \big), \label{th-bd1}
\end{align}
where $U_j$ varies over ``primitive'' integral vectors in $\z^n$, and
where $v_j$ are the diagonal elements of $\mcr Y$. We look at each sum over $U_j$. More generally, it is easy to check that for any number $a>0$,
\begin{align}
    \sumn_{X\ne 0} \exp(- a \cdot X^t X) \ll_n a^{-n/2-\epsilon},
\end{align}
because we can use the inequality $e^x>_{n,\epsilon}1+x^{n/2+\epsilon}$ for $x>0$ and note that $\displaystyle \sum_{X\ne 0} (1+X^t X)^{-n/2-\epsilon} \ll_{\epsilon,n} 1$.
Since $\mcr Y$ is also reduced, in effect, we get from \eqref{th-bd1} that
\begin{align} \label{th-bd-det}
    H_{\Gamma_n}(1_n, c_n \mcr Y)  \ll_{\epsilon,n}  \det(\mcr Y)^{-n/2-\epsilon},
\end{align}
thereby finishing the proof. \QEDB

\subsection{Proof of Theorem~\ref{mainthm-wt} for general congruence groups} \label{k-cof}

Let $\Gamma \subset \Gamma_n$ be one such congruence subgroup. If the left coset representatives are $\{\gamma_j\}$, then a fundamental domain can be chosen to be $\mc F_\Gamma=\cup_j \gamma_j \mc F^{(n)}_1$. Therefore it is enough to bound $\mbb B_k(\gamma_j(Z))$ for all $j$ for $Z \in \fn$. In other words we have to bound
\begin{align} \label{conjugate}
    \sumn_{F \in \mathscr  B(\Gamma) } \det(Y)^k |F|_k \gamma_j(Z)|^2 = \sumn_{H  \in \mathscr  B(\gamma_j^{-1} \Gamma \gamma_j) } \det(Y)^k |H(Z)|^2 ,
\end{align}
for $Z \in \fn$. The equality in \eqref{conjugate} follows since $F \mapsto F| g$ is an isometry from $\skngam$ to $\skngamconj$ for any $g \in \Gamma_n$, and $g=\gamma \gamma_j$ for some $j$ and $\gamma \in \Gamma$. 

\begin{rmk}
We note here that the  for any $g \in \Gamma_n$, one has $\sup(\skngam)=\sup(\skngamconj)$.  
\end{rmk}

We will require an elementary lemma in many of our arguments in  the rest of the paper.
\begin{lem} \label{passage}
    Let $A \subseteq B$ be subgroups of $\Gamma_n$ and $A_1 \subset A$, $B_1 \subset B$ be subgroups such that $A \cap B_1=A_1$. Then one has the natural inclusion $A_1 \backslash A \hookrightarrow B_1 \backslash B$.
\end{lem}

We would now check that the same idea of using the Fourier expansion of the Bergman kernel and bounding the Poincar\'e series on $\fn$ as in the previous Subsection also works here. We next look at the Fourier expansion of $B_{k,G}(Z,W)$ for every conjugate $G$ of $\Gamma$. Recall from \eqref{pt-exp} the definition of the Poincar\'e series $P_{T,G}$ for $G$ at the cusp $\infty$.

If we put $G= \gamma_j^{-1} \Gamma \gamma_j $, then we can apply \lemref{passage} with $A=G, A_1=G_\infty, B=\Gamma_n, B_1=\Gamma_{n,\infty}$. Indeed Lemma~\ref{passage} applies in this setting since  $G \cap \Gamma_{n,\infty}=G_\infty$.
We fix a set of coset representatives $\{g_j\}_j$ for $G_\infty \backslash G $. Then the above says that $\{g_j\}_j$ also occurs as a subset of coset representatives of $\Gamma_{n,\infty} \backslash \Gamma_n$, and we always work with this choice.

Next we take term-wise absolute values in \eqref{pt-exp-def}, use the above remark, and positivity, to obtain
\begin{align} \label{pt-G}
    P_{T;G}(Z) \ll \sumn_{g \in \Gamma_{n,\infty} \backslash \Gamma_n} |J(g,Z)|^{-k} e(- 2 \pi \tr \, \mc T Y_g /\omega).
\end{align}
Here we also use the fact that $T$ has `denominator' $\omega$, i.e. $\omega T = \mc T \in \Lambda_n$, see Section~\ref{prelim}.

Even after taking absolute values, the summands are invariant under the translation groups: namely $ \relax \im(g_\infty g(Z))=\im(g(Z))$ for $g_\infty \in G_\infty, g \in G$, therefore \eqref{pt-G} holds.

We bounded the level $1$ Poincar\'e series by taking term-wise absolute values. Hence we actually bounded the RHS of \eqref{pt-G} in \eqref{pt-bd} with the only difference that in \eqref{pt-G} we have $Y_g/\omega$ instead of $Y$. Therefore,
\begin{align} \label{ptg-bd}
    P_{T;G}(Z) \ll \big( \frac{\omega^n}{\det(Y)} \big)^{\frac{n}{2} + \epsilon}.
\end{align}
Coming back to our desired objective: we recall
\begin{align} \label{bk-pt}
 B_{k,G}(Z) = |\mcr S_G|^{-1} \, b_{n,k}^{-1} \sumn_T \det(T)^{k-\frac{n+1}{2} } P_{T;G}(Z) e(-T \overline{Z} ),
\end{align}
which gives, after using \eqref{ptg-bd}, that
\begin{align} \label{bkg-w}
     B_{k,G}(Z) &\ll  \big( \frac{\omega^n}{\det(Y)} \big)^{\frac{n}{2} + \epsilon} \,  |\mcr S_G|^{-1} \, b_{n,k}^{-1} \sumn_{\mc T \in \Lambda_G} \det( T)^{k-\frac{n+1}{2} }  \exp(- 2 \pi \mc T Y ) \\
     & \ll   k^{\frac{n(n+1)}{4}} \, \big( \frac{\omega^n}{\det(Y)} \big)^{\frac{n}{2} + \epsilon} 
\end{align}
if we invoke the bound for the $T$-sum in \eqref{bkg-w} from Lemma~\ref{bkg-lips}. Then, the relation $\mbb B_{k,G}(Z) = \frac{a_{n,k}}{2} B_{k,G}(Z)$ gives the bound 
\begin{equation} \label{bkg-1}
    \mbb B_{k,G}(Z) \ll k^{\frac{3n(n+1)}{4}}  \, \big( \frac{\omega^n}{\det(Y)} \big)^{\frac{n}{2} + \epsilon} 
\end{equation}
on all of $\fn$.
This finishes the proof of \thmref{mainthm-wt}, i.e., the $k$-aspect upper bound of the conjecture for all congruence subgroups of $\Gamma_n$, upto the proof of Lemma~\ref{bkg-lips}, which is given below. The lower bound requires a separate treatment, and can be found in Section~\ref{lbds-all}.
\QEDB

We will need the following lemmas, which is  an adaptation of \cite[Prop.~3.7]{das-krishna}  to the present situation. As far as possible, we keep the notation as it is in \cite{das-krishna} for the convenience of  the reader.

\begin{lem} \label{count-bd}
  For $Y$  Minkowski reduced, $c=c_n>0$, depending only on $n$, we can bound the size of the set
  \begin{equation}
    \#\{T\in\Lambda_G|\tr(TY)\le ck\} \ll_{n,c}
    |\mcr S_G| \cdot k^{\frac{n(n+1)}{2}}\det(Y)^{-\frac{n+1}{2}}
  .\end{equation}
\end{lem}
\begin{proof}
This is the analogue of \cite[Lemma~4.8]{das-krishna}.
Since $Y$ is reduced, we can assume that it is  diagonal: this is because $Y_D \ll_n Y$. Put $T=(t_{i,j})$, and $t_i:=t_{i,i}$ for $1 \le i \le n$. Then we have 
  \begin{align}
    \tr(TY) = \sum_i t_i y_i \ll \tr(TY) \ll k
    \implies 1 \le t_i \ll k/y_i.
  \end{align}
  Since $T>0$, when $i \neq j$, one has $|2 t_{i,j}| \le (t_i t_j)^{1/2}$. We note that since $T \in \Lambda_G$, we have $\displaystyle t_{i,j} = n_{ij}^{-1} T_{ij} $ where $T_{ij} \in \z$ for all $i,j$.
  From this, it follows that the number of choices for $T$ is, up to a constant depending only on $n$,
  \begin{equation*} 
  \ll_n
  \prod_i \frac{n_{ii} k}{y_i} \cdot \prod_{i<j} \frac{n_{ij} k}{(y_i y_j)^{1/2}} = \frac{|\mcr S_G| \cdot k^{n+n(n-1)/2}}{\det(Y)^{1+(n-1)/2}} = \frac{|\mcr S_G| \cdot  k^{n(n+1)/2}}{\det(Y)^{(n+1)/2}}
  \end{equation*} 
  if we refer to \eqref{nij-omega}. This finishes the proof.
\end{proof}

\begin{lem} \label{p-febd}
Let  $Y$ be Minkowski-reduced and put, for a congruence subgroup $G \subset \Gamma_n$,
\begin{equation}
     q_k(Y):=  \sum_{T \in \Lambda_G} p(T)^{1/2} \, \det(Y)^{k/2} \, \exp(- 2 \pi \tr(TY)) ,
\end{equation}
where
\begin{equation}
  p(T)^{1/2}  \ll_{n}
  \frac{ ( 4\pi)^{nk/2} k^{\alpha} \det(T)^{ k/2-\beta} }
  {\Gamma(k)^{n/2}}.
\end{equation}
Then we have the bound (for any $\epsilon>0$):
  \begin{equation} \label{n/2bd}
    q_k(Y)
    \ll_{n, \epsilon} |\mcr S_G| \cdot \left( 
    \left(\frac{k^{n}}{\det(Y)}\right)^{(n+1)/2 -\beta}  \cdot k^{\frac{n}{4}+\alpha} +
    \exp(-c_0 k^\epsilon) \det(Y)^{-(n+1)/2}  \right),
  \end{equation}

\end{lem}

\begin{proof}
The proof is based on showing how one can adapt the arguments given in \cite{das-krishna} to the present situation.
We will indicate how to modify \cite[Prop.~3.7]{das-krishna} to this effect. It relies on two inputs: 

(i) The analytic part, via which we truncate the sum over $T$ to a finite support via the growth of the Gamma function and decay of the exponential factor in \eqref{n/2bd}, by  analyzing the size of the eigenvalues of the matrix $TY$ with respect to $k$; and

(ii) Bound the finite support by counting the number of $T$ in the support.

Of these, a moment's reflection into the proof of \cite[Prop.~3.7]{das-krishna} presented in loc. cit. shows that it carries over mutatis mutandis. It requires no specific information about $T$, whose support is the only different thing in this whole exercise. We get
    \begin{equation} \label{eq:febd}
     q_k(Y)
    \ll_n \left( 
      \sum_{T \in \mc C(Y)}{1} + \sum_{T \notin \mc C(Y)}{\prod_j M(\lambda_j)}
    \right)            
    k^{\alpha+\frac{n}{4}}\det(T)^{-\beta},
  \end{equation}  
where $\lambda_j$ are the eigenvalues of $TY$, $\displaystyle M(x) := \left(\frac{4\pi x}{k}\right)^{k/2} \exp( (k-4\pi x)/2)$; and
the subset $\mc C(Y)$ of $\Lambda_G$ is defined by
  \begin{align} \label{cydef}
    \mc C(Y):=
    \left\{ T\in\Lambda_G \midmid \text{ all eigenvalues of } TY
    \text{ are of magnitude } \frac{k}{4 \pi} + O(k^{\frac{1}{2} +\epsilon})
    \right\}
  .\end{align}
The contribution from the first sum is as it is in \cite{das-krishna}, see the lines in between (4.19) and (4.20) in loc. cit.. Thus the first sum is
\begin{align} \label{cysum}
    \sum_{T \in \mc C(Y)}k^{\alpha+\frac{n}{4}}\det(T)^{-\beta}
    \ll_{n} \, \# \mc C(Y) \cdot
    \det(Y)^{\beta} k^{-n \beta}
    k^{\alpha+\frac{n}{4} }
  .\end{align}
Now we treat this as a counting problem on $T$. For the first sum we simply note that $\mc C(Y)$ is contained in the set
  $\{ T \mid \tr(TY) \asymp \frac{nk}{4\pi} \}$ where the implied
  constant may depend only on $n$, and invoke Lemma~\ref{count-bd}. Using the bound for $\# \mc C(Y)$ from Lemma~\ref{count-bd}, we finally get
\begin{align}
    \sum_{T \in \mc C(Y)}k^{\alpha+\frac{n}{4}}\det(T)^{-\beta}
    \ll_{n} \,|\mcr S_G| \cdot 
    \left(\frac{k^{n}}{\det(Y)}\right)^{(n+1)/2 -\beta}  \cdot k^{\frac{n}{4}+\alpha} ,
\end{align}
    which matches with the first term in the statement of Lemma~\ref{p-febd}. 

    The treatment of the ``tail'' part, i.e. the second sum $\sum_{T \notin \mc C(Y)} \cdots$ in \eqref{eq:febd} is now verbatim the same as that in \cite{das-krishna}. We have 
    \begin{equation} \label{eq:expdecay}
  \sumn_{T \notin \mc C(Y)}{\prod \nolimits_j M(\lambda_j)}
    \cdot k^{\alpha+\frac{n}{4}}\det(T)^{-\beta}
    \ll_n |\mcr S_G| \cdot \exp \left( - c_0 k^{\epsilon} \right)\det(Y)^{-\frac{n+1}{2}}
  \end{equation}
    for some $c_0>0$ depending only on $n$. This is because the only modification required in the dyadic sum treatment to prove \eqref{eq:expdecay} is, again, the content of Lemma~\ref{count-bd}. This finishes the proof.
\end{proof}

\begin{lem} \label{bkg-lips} 
With the notation as above (cf. \eqref{bkg-w}), one has the bounds
\begin{align} \label{gam-lat}
 \mc S_G:=   |\mcr S_G|^{-1} \, b_{n,k}^{-1} \sumn_{T \in \Lambda_G} \det(T)^{k-\frac{n+1}{2} } \exp(- 2 \pi \, T Y ) \ll k^{\frac{n(n+1)}{4}} .
\end{align}

\end{lem}

\begin{proof}
Recall the definition of $b_{n,k}$ from \eqref{bnk}, which is a structural constant and is the same as in \cite{das-krishna}.
 Then quoting \cite[eq.~(6.10)]{das-krishna}, with $\ell = 2k-n-1$, we have (noting that $\mc S_G=\mc S_k$ in \cite{das-krishna} with the lattice, and parameters replaced suitably)
\begin{align} \label{sg-sum}
    |{\mc S_G \det(Y)^{l/2}}| \asymp_n |\mcr S_G|^{-1} \,
  \sumn_{T\in\Lambda_G} \frac{(4\pi)^{nl/2} \ell^{\frac{-n(n+2)}{4}} \det(TY)^{\ell/2} e^{2\pi\tr(TY)}}
  {\Gamma(l)^{n/2}}.
\end{align}
We smiply invoke Lemma~\ref{p-febd} with $k$ replaced by $l=2k-n-1$, $\alpha=\frac{-n(n+1)}{4}$ and $\beta=0$ (cf. \cite[Proof of Lemma~6.2]{das-krishna}).
This implies that  
\[ |{\mc S_G \det(Y)^{l/2}}| 
\ll |\mcr S_G|^{-1} \left( |\mcr S_G|  \det(Y)^{-(n+1)/2}(k^{\frac{n(n+1)}{4}}+\exp(-c_0k^{-\epsilon})  \right) ,
\]
from which the required bound $S_G  \ll k^{\frac{n(n+1)}{4}}$ drops out.
\end{proof}

\begin{rmk}
One could have, at the expense of a bit worse dependence on the level, argued simply as follows. Put $\mc T=\omega T$. One can use \eqref{lip-1}, \eqref{lip-bd} to get
    \begin{align} \label{bkg-w1}
     B_{k,G}(Z) &\ll |\mcr S_\Gamma|^{-1} \, \omega^{\frac{n^2}{2}+ \frac{n(n+1)}{2}- nk+\epsilon} \, b_{n,k}^{-1} \sumn_{\mc T \in \Lambda_n} \det(\mc T)^{k-\frac{n+1}{2} }  \exp(- 2 \pi \mc T Y/\omega ) \\
     & \ll |\mcr S_\Gamma|^{-1} \, \omega^{n^2+\frac{n}{2} - nk+\epsilon} \, k^{\frac{n(n+1)}{4}} \det(Y/\omega)^{-k},
\end{align}
leading to the bound $\displaystyle \mbb B_{k,G}(Z) \ll k^{\frac{3n(n+1)}{4}}  \, \omega^{n^2+\frac{n}{2} +\epsilon}$.
\end{rmk}

\begin{rmk} \label{k-arith}
We now let $\Gamma$ be an `arithmetic' subgroup of $\Gamma_n$. Recall that these are by definition groups that are commensurable with $\Gamma_n$ -- i.e., $\Gamma \cap \Gamma_n$ has finite index in both $\Gamma$ and $\Gamma_n$.
We claim that \thmref{mainthm-wt} should hold for such $\Gamma$ -- in the $k$ aspect, with $N$ replaced with $\omega_\Gamma$ as defined in \eqref{nij-omega}. 

This is because any such group must have parabolic elements: if we let $S= \psmb 1_n & 1_n\\ 0_n & 1_n \psme \in \Gamma_n$, then by the pigeonhole principle some two distinct powers $S^m$ and $S^n$ must be contained in the same coset of $\Gamma \cap \Gamma_n$ in $\Gamma$, whence the assertion. Thus there is a theory of Poincar\'e series for this $\Gamma$, as treated in this paper, and since we require no information on its Fourier expansion nor that of the boundary configuration of $\Gamma$, the results should go through. Note that one can choose a fundamental domain for $\Gamma$ by $\mc F_\Gamma \subset \mc F_{\Gamma \cap \Gamma_n}= \cup_g g \fn$, where $g$ runs over $(\Gamma \cap \Gamma_n )\backslash \Gamma_n$.
\end{rmk}

\section{Congruence subgroups: proof of \thmref{mainthm-prin}}
Our aim is to prove the following statement, from which \thmref{mainthm-prin} follows trivially.
\begin{thm} \label{mainthm-rest}
Suppose $k \ge 2n+2$, and that for all the (finitely many) conjugates $G$ of $\Gamma$ of level $N$, the following conditions hold.

(i) $G \subset \Gamma^{(2)}_0(\omega_G)$,

(ii) the index  of $\mcr U_{\Gamma^{(n)}(N)}$ in $\mcr U_G$ (cf. notation in \eqref{gln-gam}) is $O_n(1)$.

Then $\sup(\skgam) \ll_n k^{\frac{3n(n+1)}{4}}$.
\end{thm}

Since our assumption about the `maximal' width of the cusps is the same for all these conjugates, it is enough to just work with $\Gamma$ and then apply it to its conjugates. Our basic strategy will be the same as that of \thmref{mainthm-wt}. But we will treat the estimate of $P_{T,\Gamma}$ differently. Namely we will isolate the contribution of those $g \in \Gamma_{\infty} \backslash \Gamma$ in the definition of $P_{T,\Gamma}$ for which $c_g=0$ and handle the rest separately. On these elements $|J(g,Z)|$ equals the minimum value $1$ as $Z$ varies in $\fn$, so the factors $J(g,Z)^{-k}$ attain their maximum value.
We state the necessary ingredients in two lemmas.

\begin{lem} \label{cgnot0-lem}
Let $\Gamma \subset \Gamma_n$ be a congruence subgroup with maximal cusp-width $\omega=\omega_\Gamma$. Suppose that $\Gamma \subset \Gamma^{(2)}_0(\omega)$.
For any $A \ge 0$ depending only on $n$, and for $k > 2n+1 + A$,
\begin{align}
    P_{T;\Gamma}(Z) = \sum_{g \in \Gamma_{\infty} \backslash \Gamma, \, c_g=0} e(T g(Z)) + O_A(\omega^{-A}), \q \q (Z \in \fn).
\end{align}
\end{lem}

\begin{proof}
    We write $T = \omega \mc T$ with $\mc T \in \Lambda_n$.
    It is enough to consider the terms with $c_g \neq 0$. For these terms we have the following:
\begin{align} 
     \sum_{c_g \ne 0, \, g \in \Gamma_{0,\infty} \backslash \Gamma}  J(g,Z)^{-k} \sum_{U \in \mcr U_\Gamma}  e(\frac{1}{\omega}\mc T[U] g(Z) )  \le   \sum_{c_g \ne 0, \, g \in (\Gamma_n)_{0,\infty} \backslash \Gamma_n}  |J(g,Z)|^{-k} \sum_{U \in \glnz}  e(- \frac{2 \pi}{\omega} \, \mc T[U] Y_g ) ,   \n 
\end{align}     
where we have used the embedding from \lemref{passage} and have put $Y_g:=\im g(Z)$. The condition $c_g \neq 0$ is stable under the aforementioned embedding. If we look at \eqref{u-gamma}, then the expression in the RHS of the above inequality is simply bounded by
\begin{align}
    &   \sum_{c_g \neq 0} |J(g,Z)|^{-k}  H_{\Gamma_n}(\mc T, Y_g/\omega) \\
    & \ll \max\{ |J(g,Z)|^{-A} \colon c_g \neq 0 \} \sumn_{  c_g \neq 0} |J(g,Z)|^{-k+A }   H_{\Gamma_n}(\mc T, Y_g/\omega) \\
    & \ll  \max\{ |J(g,Z)|^{-A} \colon c_g \neq 0 \} \, \omega^D \, \sumn_{  c_g \neq 0} |J(g,Z)|^{-k+A +D}  ; \label{g-not0}
\end{align}
where in \eqref{g-not0}, we have applied \propref{hty-lev1} with $\mcr Y=Y \gg 1$, and take $D>n/2$ depending only on $n$. We would choose $A,D$ as a functions of $n$ only.

Since $Z \in \fn$, the sum over $g$ converges and is $O(1)$ as soon as $k-A-D>n+1$ (cf. \cite[p.~24]{das-krishna}).
Next, we look at each summad in \eqref{g-not0}. For these summands, since we have $c_g \ne 0$, i.e. $\mrm{rank}(c_g) \ge 1$, and $Y \gg 1$, the idea is that, in essence, we can 
`extract' a power of $N$ from the one we factorized by $J(g,Z)$. More precisely, let $\mrm{rank}(C)=r \ge 1$. Then from Siegel \cite{siegel1935analytic} (see also \cite[Lemma~3.1]{das2015nonvanishing})
  we can  find $U,W \in \glnz$ and $C_1, D_1 $ of size $r$ such that
  \begin{align}
    (UC, UD) = \left(
    \begin{pmatrix} C_1 & 0 \\ 0 & 0 \end{pmatrix} W^t ,
    \begin{pmatrix} D_1 & 0 \\ 0 & I_{n-r}\end{pmatrix} W^{-1}
    \right). \q \q (\mrm{rank}(C_1)=r)
  \end{align}
  Using this we see that 
  \begin{align}
    |\det(CZ+D)| &= \left| \det \left(
    \begin{pmatrix}C_1 & 0 \\ 0 & 0\end{pmatrix} Z_0[W]
    + \begin{pmatrix}D_1 & 0 \\ 0 & I_{n-r} \end{pmatrix}
    \right) \right| \\
    &= \left|\det(C_1 Z[W]_* +D_1)\right|,
  \end{align}
  where $Z[W]_*$ is the leading $r \times r$ sub-matrix of $Z[W]$. Therefore,
  \begin{align}
      |\det(CZ+D)| \ge \det(C_1) |\det(Z[W]_* +C_1^{-1} D_1)| \gg \det(C_1) \ge \omega^r,
  \end{align}
since $C_1^{-1} D_1$ is symmetric, $\displaystyle Z[W]_*=X[W]_*+i Y[W]_* $, and with $\Omega:=(Y[W]_*)^{1/2} $,
\begin{align}
|\det(Z[W]_* +C_1^{-1} D_1 )| = \det(Y[W]_*)\, |\det \big(1_r+i \, \Omega^t ( X[W]_* +C_1^{-1} D_1) \Omega \big)| \ge  \det(Y[W]_*) \gg 1.
\end{align}
This implies that in \eqref{g-not0}, we have
\begin{align}
   \omega^D \, \max\{|J(g,Z)|^{-A} \colon c_g \neq 0 \} \le \omega^{D-A}.
\end{align}
The assertion of the lemma is now obvious, if we choose $D=n/2+\epsilon$, and $A=D+A'$ for $A' \ge 0$. The implied constant will depend at most on $A'$ since $|J(g,Z)| \ge 1$ as $Z \in \fn$.
\end{proof}


\begin{lem} \label{cg=0-lem}
Let $\Gamma \subset \Gamma_n$ be a congruence subgroup with maximal cusp-width $\omega=\omega_\Gamma$. Let $T \in \Lambda_\Gamma$. Suppose that the index of $\mcr U_{\Gamma^{(n)}(N)}$ in $\mcr U_G$ is $O_n(1)$. Then,
    $ \displaystyle H_{\Gamma}(T, Y/\omega) \ll 1$ for all $Y \gg_n 1$.
\end{lem}

\begin{proof}
    First, let us note the subtle point here: the matrix $\mc T=\omega^{-1}  T$ can not be assumed to be Minkowski-reduced under $\glnz= \mcr U_{\Gamma_n}$, and we can only assume that it is reduced under $\mcr U_\Gamma$. The space of $\mcr U_\Gamma$ reduced matrices may be described as an union of the form $\cup_j U_j \mcr R_n$, where $\mcr R_n $ is the space of Minkowski-reduced matrices, and $\glnz = \cup_j \mcr U_\Gamma U_j$.

    But there exists some $W \in \glnz$ such that $\mc T=T_0[W]$, such that $T_0$ is reduced. Then,
    \begin{align} \label{hgam}
H_{\Gamma}(T, Y/\omega) = \sum_{U \in \mcr U_\Gamma} \exp(- \frac{2 \pi}{\omega} \tr \, \mc T  Y[U] ) \ll \sum_{U \in \mcr U_\Gamma} \exp(- \frac{1}{\omega} \tr \, T_0[WU] ) \ll 
\sum_{U \in \mcr U_\Gamma} \exp(- \frac{1}{\omega} \tr \,[WU] )
    \end{align}
since $T_0$ being reduced, satisfies $T_0 \gg \mrm{diag}(T_0) \gg 1_n$.

    Since the index of $\mcr U_{\Gamma^{(n)}(N)}$ in $\mcr U_G$ is $O_n(1)$, each $U \in \mcr U_\Gamma$ is congruent to finitely many matrices, say, $\{A_1,\ldots, A_t\}$, $t=O_n(1)$. Thus in \eqref{hgam}, we can split the sum over $U$ as $t$ sums of the form $\sum_{U \equiv W A_j \bmod N}$ where $1 \le j \le t$. It is enough to consider each such sum, and replace $WA_j$ by $W$.
    
    Next, we can assume without loss that $W$ has entries in $\{ 0,1, \ldots, N-1\}$, by considering $U_*W$, with $U_* \in \mcr U_{\Gamma^{(n)}(N)} = \{ V \in \glnz \mid V \equiv 1_n \bmod N \}$ -- the principal congruence subgroup in $\glnz$. This is possible since $U_{\Gamma^{(n)}(N)} \subset \mcr U_\Gamma$, and since for $W,V \in \glnz$, one has $U_*W=V$ for some $U_* \in U_{\Gamma^{(n)}(N)}$ if and only if $W \equiv V \bmod N$. Let us fix such a $W$.

    One can thus rephrase \eqref{hgam} as
\begin{align} \label{hgam1}
    H_{\Gamma}(T, Y/\omega) \ll  \sum_{U \in \glnz, \, U \equiv W \bmod N} \exp(- \norm{U}^2/ \omega) = \sum_{U \in \glnz, \, U \equiv W \bmod N} \exp(- \frac{1}{\omega} \, \sum_{ij} u^2_{ij}  ),
\end{align}
with $\norm{U} = (\tr \,  U^tU)^{1/2} $, and we have to show that this is $O(1)$. Write $u_{ij}=w_{ij} + b_{ij}N$ for some integers $b_{ij}$. Unless for all pairs $i,j$ one has $b_{ij}=0,-1$, we see that $U^tU \ge N^2$. The number of exceptional $U$, for which the above is not satisfied, can be estimated as $O(2^{n^2})$. Recall that $W$ is fixed. For the exceptional $U$ we simply bound the exponential by $1$ and for the rest we notice that with $x= \norm{U}^2$, one has $x \ge N x^{1/2}$. Using these in \eqref{hgam1}, we get, 
\begin{align} \label{hgam2}
    H_{\Gamma}(T, Y/\omega) \ll 1+ \sumn_{U \in \glnz} \exp(- \frac{N}{\omega} \norm{U} ) \ll 1+ \sumn_{U \in \glnz} \exp(-  \norm{U} ) \ll 1.
\end{align}
This finishes the proof of \lemref{cg=0-lem}.
\end{proof}

\lemref{cgnot0-lem} and \lemref{cg=0-lem} together show the following result, which is the backbone of this paper.
\begin{prop} \label{poin-bd}
    In the setting of \thmref{mainthm-rest}, one has the uniform bound 
    \begin{align} 
    \displaystyle  P_{T; \Gamma}(Z) \ll_n 1, \q \q (Z \in \fn)
    \end{align}
    for all $T \in \Lambda_\Gamma$ and $k \ge 2n+2$.
\end{prop}

\begin{rmk} 
The reader may wonder why we can't apply \lemref{cg=0-lem} to the terms with $c_g \neq 0$. The reason is that $Y_g$ in that case may not be bounded below, and the argument only would give $H_{\Gamma}(T, Y_g/\omega) \ll \prod_j (1+\frac{1}{v_j^D})$, for some $D>0$, where $v_j$ are the diagonal elements of $Y_g$. We do not know how to handle this expression. We cannot assume that $Y_g$ is $\glnz$ reduced (cf. proof of \lemref{cg=0-lem}). One can check that $ \displaystyle Y_g^{-1}=(c_gX+d_g)^tY^{-1}(c_gX+d_g)+ c_g^tYc_g$, for which one needs a suitable upper bound. One may use the Frobenius norm, which will lead one to consider a sum of the type 
\[ \sumn_{ \{ C,D\} } \big( \norm{CX+D}^{2d}+\norm{C}^{2d} \norm{Y}^d\big) |\det(CZ+D|^{-k} \]
for some $d>0$, where $\{C,D \}$ run over co-prime symmetric pairs. This will anyway be a polynomial in $Y$ which is not useful for us.
\end{rmk}

\subsection{Completion of the proof of \thmref{mainthm-rest}}
If we follow the algorithm from \ref{gen-str}, we see that only Step (2) needs to be checked (Subsection~\ref{step3} proves Step~(3)) -- which however is proved in \propref{poin-bd}. Lower bounds have been worked out in Section~\ref{lbds-all}. \QEDB

\section{Quantitative asymptotic formula for the Fourier coefficients of Siegel Poincar\'e series and applications} \label{poin-sec}
The aim of this Subsection is to provide a quantitative formula for the Fourier coefficients of Siegel Poincar\'e series $P_{\Gamma,T}$ on arbitrary congruence subgroups $\Gamma \subset \Gamma_n$. This will have two consequences: first, in Subsection~\ref{lbd-sec} (in the next Section) we would apply this to get the desired lower bound for the Bergman kernel of $\skgam$; and second, we shall be able to conclude that for certain $T>0$, one has $P_{\Gamma,T} \neq 0$.

For this, we would continue to use the ideas in this paper, adapt some of the arguments presented in \cite{das-krishna}, and quantify them suitably. Let $p_{T} (T')=p_{T, \Gamma} (T')$ denote the $T'$-th Fourier coefficient of $P_{\Gamma,T}$. We suppress the $\Gamma$ from the notation, as it is understood.
The aforementioned arguments in turn relied on an asymptotic  result \cite[Proof of Thm.~4]{kst} about $p_{1_n} (1_n)$ in degree $n \ge 1$. Namely, for $\Gamma=\Gamma_n$ and $T \in \Lambda_n$ one has
\begin{align} \label{lim-po}
    \lim_{k \to \infty} p_{T} (T) = |\mrm{Aut}(T)|/2 ;
\end{align}
where $\mrm{Aut}(T)$ denotes the number of automorphisms of $T$ under $U \in \glnz$. We would generalize this to any congruence subgroup $\Gamma$ and in a quantitative sense.

\begin{thm} \label{pt-asymp}
Let $\Gamma \subset \Gamma_n$ be any congruence subgroup which satisfies either

(i) the width of the cusp $\omega_\Gamma$ at $\infty$ for $\Gamma$ is $1$; or,

(ii) $\Gamma^{(n)}(N) \subseteq \Gamma \subseteq \Gamma^{(n),0}_0(N)$

Let $y_0>1$ be arbitrary in case of (i) and put $y_0=Ny_1$ with $y_1>1$ being arbitrary in case of (ii). 

Let $P_T:=P_{\Gamma;T}$ be the $T$-th Poincar\'e series on $\Gamma$, with Fourier coefficients $p_T(T')$. For $k \ge 2n+6$, one has one has an asymptotic formula
    \begin{align}
  p_{T}(T')  = 
 \frac{|\mrm{Aut}_\Gamma(T';T)|}{2} + O\Big(\exp( 2 \pi y_0 \, \tr (T')) \, (\frac{\omega_\Gamma}{y_0})^{(n^2+1)/2} y_0^{- k/2n} \Big)  ,
\end{align}
where the implied constant depends only on $n$. 

\end{thm}

In the setting of Theorem~\ref{pt-asymp}, let us put
\begin{align}
  \Gamma'=  \begin{cases}
        \Gamma_n, & \text{in the case (i)}\\
       \Gamma^{(n),0}_0(N) , & \text{in the case (ii)}
    \end{cases}.
\end{align}
We also note the following.
\begin{align}\label{transl-cases}
  \omega_\Gamma=  \begin{cases}
        1 \\ N
    \end{cases};  \q
 \mcr U_\Gamma =   \begin{cases}
          \glnz \\ \glnz \end{cases};
         \q \mcr S_\Gamma =   \begin{cases}
          \symn & \text{in the case (i)}\\
         N \cdot \symn  & \text{in the case (ii)}
    \end{cases}.
\end{align}

We need the following lemma to prove Theorem~\ref{pt-asymp}. For $s>n+1$ and $\mc Z \in \hn$ and a congruence subgroup $\Gamma \subset \Gamma_n$, put 
\begin{align} \label{mgamma'}
 \mcr M(\Gamma; \mc Z)   := \sum_{g \in (\Gamma)_{0,\infty}\backslash \Gamma} |\det(c_g (Z_0+S)+d_g)|^{-s}
\end{align}

\begin{lem} \label{Mgam-lem}
    For any congruence subgroup $\Gamma \subset \Gamma_n$,
     any $s>n+1$, $\mc Z \in \fn$, one has \[ \mcr M(\Gamma; \mc Z) \ll_{n,s} 1.\]
\end{lem}

\begin{proof}
First of all note that the coset space $(\Gamma)_{0,\infty}\backslash \Gamma$ is parametrized by all the ``in-equivalent coprime symmetric pairs" belonging to $\Gamma$ -- i.e., all the pairs $(C,D)$ which occur as the last row-block of elements in $\Gamma$, up to the left action by $\mcr U_\Gamma$: $(C,D) \to (UC,UD)$, $U \in \mcr U_\Gamma$.

We will use the following result shown in \cite[see the paragraphs after (6.28)]{das-krishna}:
for any $\mc Z \in \fn$ and any $s>n+1$, one has
\begin{align}\label{ezs}
 \mcr M(\Gamma_n, \mc Z) =  \sum_{(C,D)} |\det(C \mc Z+D|^{-s} \ll_s 1,
\end{align}
where $(C,D)$ vary over ``in-equivalent coprime symmetric pairs" belonging to $\Gamma_n$.
Moreover, the implied constant in \eqref{ezs} is absolute -- does not depend on $\mc Z$.

We next employ Lemma~\ref{passage} to note that the natural map $(\Gamma)_{0,\infty}\backslash \Gamma \hookrightarrow (\Gamma_n)_{0,\infty}\backslash \Gamma_n$ is injective. To see this note the following: the equality 
\begin{align}
\pmat{U^t}{0}{0}{U^{-1}} \pmat{*}{*}{c_1}{d_1}  =  \pmat{*}{*}{c_2}{d_2} ;
\end{align}
with $U\in \glnz, \smat{*}{*}{c_j}{d_j} \in \Gamma$ for $j=1,2$ implies that $U \in \mcr U_{ \Gamma}$. This means that we can apply Lemma~\ref{passage}, and this finishes the proof.
\end{proof}

\begin{lem} \label{y0lem}
    Let $A \ge 0$. Let us first observe that for any $y \in \R, y>1$ , $Z=X+iy 1_n$, any $X \in \symnr$  and any $\gamma \in \Gamma^{(n)}_0(N)$ with $c_\gamma \neq 0$, 
\begin{equation} \label{yo}
    |\det(c_\gamma Z+d_\gamma)|^{-A} \le (Ny_0)^{-A},
\end{equation}
\end{lem}

\begin{proof}
    This is from \cite[Lem.~3.2]{das-anamby1}, which improves the result in \cite[Lem.~5]{kst}, where \eqref{yo} was obtained for \textit{some} $y_0>1$. See also the proof of Lemma~\ref{cgnot0-lem}, in which  the essential idea is captured. We do not reproduce the proof, which will be verbatim. Instead, we simply notice that \cite[Lem.~3.2]{das-anamby1} holds for all $X=\re(Z)$ -- the proof only uses the imaginary parts all through. The $N$ in the bound on the RHS of \eqref{yo} comes from the condition $N|c_\gamma$.
    This gives the lemma.
\end{proof}

\begin{proof}[Proof of Theorem~\ref{pt-asymp}]
To prove this, we would appeal to some of our earlier calculations in Section~\ref{k-cof}. We begin with:
\begin{align}
    P_T(Z):= P_{T,\Gamma}(Z)=P_o(Z)+P_*(Z),
\end{align}
where $P_o(Z)$ is the contribution of all terms with $c_g=0$ in \eqref{pt-G-lbd}, and the rest is denoted by $P_*(Z)$. 

Note that in case (i),  $\omega_\Gamma=1$ and in case (ii), we have 
$\omega_\Gamma=N=$ the level of $\Gamma$.

We specialize to $Z=Z_0:= X+ iy_01_n$, with $X \bmod \mcr S_\Gamma$ and $y_0$ as in the statement of the Theorem. 

With this combined notation, we thus have, 
\begin{align} \label{pt-G-lbd}
    P_{o}(Z_0) &= \sumn_{g \in \Gamma_{\infty } \backslash \Gamma \colon c_g = 0 } e( T g(Z_0) ) = \sumn_{U \in \mcr U_\Gamma} e( (T Z_0[U])  ;\\
    P_{*}(Z_0) &\le  \sumn_{g \in \Gamma'_{\infty } \backslash \Gamma' \colon c_g \neq 0 }  |J(g,Z_0)|^{-k} e \big(- 2 \pi \tr \,  T \im(g(Z_0))  \big) \\
    &=\sumn_{g \in (\Gamma')_{0,\infty} \backslash \Gamma' \colon c_g \neq 0 }  |J(g,Z_0)|^{-k} H_{\Gamma'}( T, \im(g(Z_0))  \\
    & \le \sumn_{g \in (\Gamma')_{0,\infty} \backslash \Gamma' \colon c_g \neq 0 }  |J(g,Z_0)|^{-k} H_{\Gamma_n}( T, \im(g(Z_0))
\end{align}
where, as before (cf. the discussion around \eqref{pt-G}), we have used the embedding $\Gamma_\infty \backslash \Gamma \hookrightarrow \Gamma'_{\infty} \backslash \Gamma'$ for the first inequality. Also we have used \eqref{hgamma-ineq} in the last inequality.

Note that the condition $c_g =0$ (or $c_g \neq 0$) is invariant under left multiplication by $\Gamma_\infty$.
Next, let us record the formula for Fourier coefficient which reads as follows:
\begin{align} \label{fc-def}
\exp(- 2 \pi y_0 \, \tr (T')) \, p_T(T') &= |\mcr S_\Gamma|^{-1} \int_{X \bmod \mcr S_\Gamma} P_{T}(Z_0) e(-T' \, X)dX .
\end{align}

We next want to show that $P_*(Z_0)$ is small as soon as $k$  is large enough. To see this, we have to estimate $H_{\Gamma_n}( T, \im(g(Z_0))$. Towards this, we apply \propref{hty-lev1} with $\mcr Y= \im(g(Z_0))/\omega$, so that $\displaystyle \det(\mcr Y) = (\frac{y_0}{\omega})^n |J(g,Z_0)|^{-2}$ and $D=n/2+\epsilon$ (loc. cit.). Here we have used that $\omega T\in \Lambda_n$. We get,
\begin{align} \label{p*bd}
    P_*(Z_0) \ll (\frac{\omega}{y_0})^{n^2/2+ n\epsilon} \sumn_{g \in (\Gamma')_{0,\infty}\backslash \Gamma' \colon c_g \neq 0} \frac{1 }{|\det(c_g Z_0+d_g)|^{k-n- 2\epsilon} }.
\end{align}

Now put $B:=k-n-2\epsilon$, $B':=B-k \epsilon=k-n -(k+2) \epsilon$. $0<\epsilon<1$ will be chosen later as a function of $n$.
With this in view, we can write, 
\begin{align}
    P_*(Z_0) & \ll (\frac{\omega}{y_0})^{n^2/2+n \epsilon} \max_{g \in \Gamma_n} \{  |\det(c_g Z_0+d_g)|^{-k \epsilon} \colon c_g \neq 0   \} \, \sum_{g \in (\Gamma')_{0,\infty}\backslash \Gamma' \colon c_g \neq 0} |\det(c_g Z_0+d_g)|^{-B'} \\
    & \ll  (\frac{\omega}{y_0})^{n^2/2+n \epsilon} (Ny_0)^{-k \epsilon} \cdot \mcr M(\Gamma', Z_0) ,\label{y0-eps}
\end{align}
where in the last inequality we have used Lemma~\ref{y0lem}.

CLAIM: $\mcr M(\Gamma', Z_0) \ll_\epsilon 1$.

In the case (i), this is a direct application of Lemma~\ref{Mgam-lem} since here $Z_0 \in \fn$. Recall that $X$ varies $\mod \mcr S_\Gamma$, i.e. $\mod 1$ in this case.

In the case (ii), we argue as follows. In this case $X$ varies $\mod N$, see \eqref{transl-cases}. We can thus write 
\[ c_g Z_0+d_g=Nc_g Z_1+d_g, \q Z_0=X+iy_0 1_n= N(X_1+iy_1 1_n)=NZ_1,\]
where now $X_1$ varies $\mod 1$ and $y_0=Ny_1$ as in the statement of Theorem~\ref{pt-asymp}. Thus $Z_1 \in \fn$.

Next note that the assignment $(c_g,d_g) \mapsto (Nc_g,d_g)$ is a realization of the injective map from the quotient $(\Gamma')_{0,\infty}\backslash \Gamma' \hookrightarrow (\Gamma_n)_{0,\infty}\backslash \Gamma_n$ defined by $\displaystyle g \mapsto \alpha_N^{-1} g \alpha_N$, where $\alpha_N=\smat{N 1_n}{0}{0}{1_n}$. This is because of the following equalities. Let $g \in \Gamma'$ and recall that $m(U)=\smat{U^t}{0}{0}{U^{-1}}$. 
\[ \alpha_N^{-1} g \alpha_N = \smat{a_g}{N^{-1} b_g}{Nc_g}{d_g}\in \Gamma_n;\q \mcr U_{\Gamma'}=\glnz;\q \alpha_N m(U) \alpha_N^{-1}=m(U)\, \forall \, U \in \glnz.  \]
Therefore
\[ \mcr M(\Gamma', Z_0) \le \mcr M(\Gamma_n, Z_1)  \ll_\epsilon 1,\]
by Lemma~\ref{Mgam-lem}, proving the CLAIM.

At this point, note that the $\epsilon$ featuring in the exponents of $\omega$ and $y_0$ are actually same. 
We will choose $\epsilon=1/(2n)$.
We have to choose $k$ large enough so that $B'>n+1$.
This means that we have to choose $k$ a bit large, e.g. $k \ge 2n+6$ is enough.

The contribution of the terms with $c_g \neq 0$ in $p_{T}(T)$ in \eqref{fc-def} is
\begin{align}
\ll |\mcr S_\Gamma|^{-1} \int_{X \bmod \mcr S_\Gamma}  P_*(Z_0)  dX \ll y_0^{- k/2n - n^2/2-1/2} N^{-k/2n}.
\end{align}
Finally we see that $\exp(- 2 \pi y_0  \tr (T')) \, p_{T}(T')$ equals
\begin{align}
 &= |\mcr S_\Gamma|^{-1}  \sum_U  \int_{X \bmod \mcr S_\Gamma}  \big( P_o(Z) + P_*(Z) \big) e(-T'X) dX \\
&=  |\mcr S_\Gamma|^{-1} \exp(- 2 \pi y_0  \tr (T')) \sum_U  \int_{X \bmod \mcr S_\Gamma} e( (T[U]-T') X) dX + O((\frac{\omega}{y_0})^{(n^2+1)/2} y_0^{- k/2n}) \\
&= \frac{|\mrm{Aut}_\Gamma(T';T)|}{2} \exp(- 2 \pi y_0  \tr (T')) + O((\frac{\omega}{y_0})^{(n^2+1)/2 \epsilon} y_0^{- k/2n}) ,
\end{align}
as desired. A comparison with \eqref{fc-def} then finishes the proof of \thmref{pt-asymp}.
\end{proof}

\begin{rmk}
    For simplicity we have chosen the maximal cusp width $\omega$ at $\infty$ to be $1$ in  the above theorem. Otherwise one most likely would get a polynomial in $\omega$ in the error term with exponent $O_n(1)$. 
    One may combine the arguments in \cite[section~6]{das-krishna} and \cite[Section~6]{das2015nonvanishing} for this. We choose not to do this here.
\end{rmk}


\subsection{Application to the nonvanishing of Poincar\'e series}
A similar argument gives the following corollary about non-vanishing of Poincar\'e series from \thmref{pt-asymp}. It is generally believed that,
\begin{conj}
   For $k>n+1, nk \equiv 0 \mod 2$, all the Siegel Poincar\'e series $P_{T, \Gamma_n}$ (of exponential type) are non-zero.
\end{conj}

Not much is known -- one has partial results from \cite{das-seng-poin}, where images of Saito-Kurokawa lifts were utilized towards this cause. This question has various avatars, applications, interpretations and  various number theoretic consequences -- we refer the reader to the nice survey \cite{zunar} on this topic. For the author, this question for Poincar\'e series of exponential type (which show up often in number theory) remains an outstanding open question in the classical theory of modular forms. For other Poincar\'e series of `polynomial' type, positive results are known -- see e.g. \cite{zunar2} for results on this topic.

\begin{cor}[Non-vanishing of Poincar\'e series (of exponential type)] \label{poin-cor}
    Let $\Gamma$ be any congruence subgroup of $\Gamma_n$ as in the statement of Theorem~\ref{pt-asymp}.  Then $P_{\Gamma;T}$ does not vanish identically for all $T$ such that 
    \begin{align}
        \tr \mc T \ll_n \begin{cases}
             k & \text{in the case (i)}\\
             k (\log N +1) & \text{in the case (ii)}.
        \end{cases}
    \end{align}
\end{cor}
\begin{proof}
    The proof is immediate from the asymptotic formula in Theorem~\ref{pt-asymp}; and the asymptotic holds in each case provided the conditions stated in the lemma hold.
\end{proof}

\begin{rmk}
    In \cite[Thm.~1.2]{das-seng-poin} it was shown that when $n=2, N=1$, $P_T \neq 0$ roughly when $\det(T) \ll k^2$ using the explicit Fourier expansion of certain Poincar\'e series of half-integral weights. Note that here one has to count only reduced $T$, as $|P_T|$ (and indeed the conditions of the theorem loc. cit.) is invariant under $\mrm{GL}_2(\z)$. Thus this result is valid for around the `first' $(k^2)^{3/2}=k^3$ reduced matrices.

    Our Corollary~\ref{poin-cor} also has the same strength as discussed above (number of reduced matrices $T \in \Lambda_n$ with $\tr(T)\ll k$ is about $k^3$), and moreover covers all degrees $n$ and all congruence subgroups with maximal cusp-width $1$ at the cusp $\infty$. The main obstacle around a quantitative version of this question is the intractability of the Fourier expansion of Siegel Poincar\'e series, which means the traditional method of, say, Rankin \cite{rankin} is not feasible. This issue is another outstanding open question -- viz., requiring good understanding of matrix argument Kloosterman sums and Bessel functions cf. \cite{herz}. The above Corollary thus generalizes the previously known nonvanishing results on Siegel Poincar\'e series in degrees $\le 2$ to all higher degrees without the use of explicit formula of its Fourier coefficients.
\end{rmk}

\subsection{Linear relations among Poincar\'e series}
The following corollary about linear relations among Poincar\'e series gives the $p=\infty$ version of the main results in \cite{boech-das}, where congruences between modular forms modulo $ p$ were used to prove the following theorem for $k$ restricted to specific families of arithmetic progressions (depending on $p$) when $n$ is even. Such progressions, even though completely explicit, however do not cover all but finitely many positive integers. Our result is valid for all large enough weights and can obviously be formulated for general congruence subgroups. We state it on $\Gamma_n$ for convenience.
\begin{cor}\label{poin-nonvanish}
    Let $\{T_1, \ldots, T_h \} \subset \Lambda_n$ be a $\glnz$-inequivalent subset. Then the Poincar\'e series $\{ P_{T_1}, \ldots, P_{T_h} \}$ on $\Gamma_0^{(n)}(N)$ are linearly independent provided $k \gg_n \max
    \{ \tr T_j \colon 1 \le j \le h \}$.
\end{cor}

\begin{proof}
Since the maximal width of the cusp at $\infty$ for $\Gamma_0^{(n)}(N)$ is $1$, we can apply the case (i) of Theorem~\ref{pt-asymp} to the Poincar\'e series in question.
We show that the matrix $\mcr P :=(p_{ij})$ with $p_{ij}:= p_{T_j}(T_i)$, $1 \le i,j \le n$, has maximal rank. For this we expand $\det(\mcr P)$ according to Leibnitz rule: except for the diagonal product $p_{11}\cdots p_{nn}$, all other terms are easily seen to be $O\big(\exp(c y_0 \max_j{\tr(T_j)}) y_0^{-(n^2+1)/2- k/2n} \big)$. We choose $y_0=2$ (any $y_0>1$ should work). For $k$ large enough as in the statement of the corollary we see that each such term is $O(2^{-k/4})$. Similarly the product $p_{11}\cdots p_{nn} = M + O(2^{-k/4n})$ for some $M \gg 1$. This finishes the proof.
\end{proof}

\corref{poin-nonvanish} generalizes the results of \cite{zhining} to higher degrees. Whereas in \cite{zhining}, explicit Fourier expansion of degree $2$ and level $1$ Poincar\'e series was used -- we don't have it when $n>2$. Also when $n=2$, it is clear that we can recover \cite[Thm.~5.1]{zhining} from the above corollary.

\section{Lower bounds for average sizes} \label{lbds-all}
To begin with, note that if $\Gamma_1 \subseteq \Gamma_2 \subseteq \Gamma_n$ be two subgroups of finite index, then one has the inequality obtained by embedding one space into the other:
\begin{equation} \label{mu12}
    \mbb B_{\Gamma_2}(Z) \le | \Gamma_2 \colon \Gamma_1| \mbb B_{\Gamma_1}(Z) ,
\end{equation}
where $| \Gamma_2 \colon \Gamma_1|$ denotes the index of $\Gamma_1$ in $\Gamma_2$. Thus we immediately obtain 
\begin{align} \label{mu-1}
  \sup(S^{(n)}_k) |\Gamma_n \colon \Gamma|^{-1} \ll_n   \sup(\skngam).
\end{align}
for any congruence subgroup $\Gamma$. 

Our next goal is to remove the index factor in the denominator (in many cases). At this point we want to (again) remind the reader about the issue of not having an usable bound on the Fourier coefficients of Siegel Poincar\'e series of higher degrees on $\Gamma_n$, let alone congruence subgroups. In the context of obtaining lower bounds on $\sup(\skngam)$, this was overcome in \cite[Sec.~7.1]{das-krishna} by using only `soft' asymptotic properties of these Fourier coefficients as $k \to \infty$. This point is relevant in what follows about getting lower bounds on $\sup(\skgam)$. See also \thmref{pt-asymp} in Section~\ref{poin-sec} in this regard, which  allows us to control the range of validity of the weight $k$ in an an asymptotic formula for the Fourier coefficients of Siegel Poincar\'e series explicitly.

In the next Subsection we prove the lower bounds claimed in \thmref{mainthm-wt}, \thmref{mainthm-prin} and \thmref{mainthm-klarge}. In the case of congruence subgroups, in the weight aspect, \eqref{mu-1} is enough. In the level aspect things are somewhat subtle, and we need more care. 

\subsection{Proof of the lower bounds via average of Fourier coefficients} \label{lbd-sec}

The proofs will be based on the following lemmas.

\begin{lem}
    For any $g \in \Gamma_n$, one has $\sup(\skngam)=\sup(\skngamconj)$.
\end{lem}

\begin{proof}
    Take $\mc F_\Gamma=\cup_j \gamma_j \mc F^{(n)}_1$ as the Siegel fundamental domain for $\Gamma$. Then it follows that
    \begin{align}
        \sup(\skngam) &= \max_{j} \sup_{Z \in \fn} \sumn_{F \in \mathscr  B(\Gamma) } ( \det(Y)^k |F|^2)(\gamma_j Z) \\
        &=\max_{j} \sup_{Z \in \fn} \sumn_{F \in \mathscr  B(\Gamma) } \det(Y)^k |F|_k \gamma_j(Z)|^2 \\
        &=\max_{j} \sup_{Z \in \fn} \sumn_{H  \in \mathscr  B(\gamma_j^{-1} \Gamma \gamma_j) } \det(Y)^k |H(Z)|^2 , \label{symm-conj}
    \end{align}
since $F \mapsto F| \gamma_j$ is an isometry from $\skngam$ to $S_k^{(n)}(\gamma_j^{-1} \Gamma \gamma_j)$ for all $j$. The expression in \eqref{symm-conj} is symmetric in all the conjugates of $\Gamma$: for any $g \in \Gamma_n$, write $g=\gamma \gamma_j$ for some $j$ and $\gamma \in \Gamma$ and notice that $g^{-1} \Gamma g = \gamma_j^{-1} \Gamma \gamma_j$. This finishes the proof.
\end{proof}

\begin{lem} \label{cgam-lbd-lem}
  \begin{align} \label{cgam-lbd}
    \sup(\skgam) \gg  \mf c_\Gamma:= \max_{G} \max_{T_G} \, |\mcr S_G|^{-1}  \, \det(T_G)^{-\frac{n+1}{2}}  \cdot p_{T_G}(T_G),
    \end{align}
    where $G$ runs over all the conjugates of $\Gamma$ and for each $G$, $T_G$ traverses $\Lambda_G$.
\end{lem}

\begin{proof}
The proof is based on the same ideas in presented in \cite[Section~7]{das-krishna}, albeit in level one. For the convenience of the reader, we recall some of the details here. The case of congruence subgroups do complicate matters to some extent.

For a congruence subgroup $\Gamma$ of cusp width $\omega$ at the cusp $\infty$, we begin with the familiar argument. Let $T_0 \in \Lambda_\Gamma$. Then, since the volume of $\symnr \mod \mcr S_\Gamma$ is the determinant $|\mcr S_\Gamma|$, by the definition of $\skgam$, we have
\begin{align} \label{lbd-1}
 \sup(\skgam)
 \ge |\mcr S_\Gamma|^{-1} \int_{X \bmod \mcr S_\Gamma} \, \sum_{F \in \mcr B(\Gamma)} \det(Y_0)^{k} |F(X+iY_0)|^2 \, dX .
\end{align}
Next, by the Cauchy-Schwarz inequality the RHS of \eqref{lbd-1} is
\begin{align}
 &\ge  |\mcr S_\Gamma|^{-1} \det(Y_0)^{k} \sum_{F \in \mcr B(\Gamma) } \, | \int_{X \bmod \mcr S_\Gamma} \frac{F(X+iY_0)}{e(\tr T_0(X+iY_0))}\, dX | ^2 \exp{(-4\pi \tr(T_0 Y_0)} \n \\
 & =  \det(Y_0)^{k} \exp{(-4 \pi \tr(T_0 Y_0))} \sum_{F \in \mcr B(\Gamma)} \, |a_F(T_0)|^2 . \label{lbd-2}
\end{align} 
We now put $Y_0:= k /(4 \pi )\cdot T_0^{-1}$, which gives us that \eqref{lbd-2} is at least
\begin{align}
& \gg  \det(T_0)^{-k} (\frac{k}{4 \pi})^k e^{- nk} \sum_{F \in \mcr B(\Gamma)} \, |a_F(T_0)|^2 \\
&= (\frac{k}{4 \pi})^k e^{- nk} c_{n,k}^{-1} \, \frac{|\mcr S_\Gamma|^{-1} \, \det(T_0)^{-\frac{n+1}{2}}}{  \Gamma_n(k-\frac{n+1}{2})} \cdot p_{T_0}(T_0), \label{lbdp1}
\end{align}
where we have used \eqref{sumsq-fc}, and we also recall from \eqref{cnk-def} the constant $c_{n,k}$, which is defined by
\begin{equation} \label{cnk}
 c_{n,k} = \pi^{n(n-1)/2} (4 \pi)^{n(n+1)/4 - nk} \prod \nolimits_{j=1}^n \Gamma(k - (n+j)/2).
\end{equation}
The $k$ dependence calculation remains unchanged as in \cite{das-krishna}, and \eqref{lbdp1}, upon using Stirling's bound for the $\Gamma$ function shows that
\begin{align} \label{T0}
    \sup(\skgam) \gg |\mcr S_\Gamma|^{-1} \, \det(T_0)^{-\frac{n+1}{2}}  \cdot p_{T_0}(T_0)\, \, k^{\frac{3n(n+1)}{4}}.
\end{align}
Now note that since $\sup(\skgam)=\sup(S_k(G))$ for any conjugate $G$ of $\Gamma$, we can deduce from \eqref{lbdp1} (applied to the conjugates) that
\begin{align}
    \sup(\skgam) \gg  |\mcr S_\Gamma|^{-1} \, \det(T_0)^{-\frac{n+1}{2}}  \cdot p_{T_0}(T_0)\, \, k^{\frac{3n(n+1)}{4}}
\end{align}
{\it for all} the conjugates $G$ of $\Gamma$. Thus the lemma follows.
\end{proof}

\begin{lem} \label{lbd-lem2}
    For all congruence subgroup $\Gamma \subset \Gamma_n$ of level $N$, when 
    
    (i) $N \asymp_n 1$; or when 
    
    (ii) $\omega_G=1$ for some conjugate $G$ of $\Gamma$, 
    
    one has $\mf c_\Gamma \gg_n 1$. 
    
    In particular, for the groups
    $\Gamma=\Gamma^{(n)}_0(N), \Gamma^{(n)}_1(N)$ we have $\mf c_\Gamma \gg_n 1$.
\end{lem}

\begin{proof}
The case (i) follows immediately from the lower bound in level one from \cite[Theorem~1.2]{das-krishna} and the inequality \eqref{mu-1}.

For the case (ii), first of all in view of the symmetry in the conjugates of $\Gamma$ in the definition of $\mf c_\Gamma$ in the definition \eqref{cgam-lbd}, its enough to assume that $\omega=\omega_\Gamma=1$. Thus we are reduced to considering \eqref{T0}, for which we require a lower bound $ \gg 1$ for a suitable choice of $T_0$.

We choose $T_0=1_n$. We already have $|\mcr S_{\Gamma}|=1$ in this case. It remains to note that $ p_{1_n}(1_n) \gg 1$ which follows from Theorem~\ref{pt-asymp}~case~(i). We take $T=T'=1_n, y_0 \asymp 1$. This finishes the proof.
\end{proof}

\begin{lem} \label{lbd-lem3}
    For all congruence subgroups $\Gamma \subset \Gamma_n$ of level $N$ such that $ \Gamma^{(n)}(N) \subseteq \Gamma \subseteq \Gamma^{(n),0}(N)$, one has one has $\mf c_\Gamma \gg_n 1$. 
\end{lem}

\begin{proof}
    The proof is the same as that of Lemma~\ref{lbd-lem2}. In this case we have $\omega_\Gamma=N$ and $\mcr S_\Gamma=N \symn$, see \eqref{transl-cases}. We work with $G=\Gamma$. Again in \eqref{T0}, we choose $T_0=N^{-1}1_n$, so that 
    \[ |\mcr S_\Gamma|^{-1} \, \det(T_0)^{-\frac{n+1}{2}} = N^{-n(n+1)/2} \cdot N^{n(n+1)/2}=1.\]

    It remains to note that $ p_{T_0}(T_0) \gg 1$ where $T_0=N^{-1}1_n$, which follows from Theorem~\ref{pt-asymp}~case~(i). We take $T=T'=T_0, y_0 \asymp N$. This finishes the proof.
\end{proof}

\subsection{Proof of the lower bounds in Proof of the lower bounds in \thmref{mainthm-wt}, \thmref{mainthm-prin} and \thmref{mainthm-klarge} }

All of the groups $\Gamma$ listed in the above mentioned theorems, satisfy one of the properties listed in the Lemma~\ref{lbd-lem2} and~\ref{lbd-lem3}. Therefore for all these $\Gamma$, one has $\mf c_\Gamma \gg_n 1$. 
This finishes the proof of our claims about the lower bounds in the theorems. \QEDB

\section{A general strategy} \label{gen-str}
We now outline a general strategy that should work for all congruence subgroups. We present it as an algorithm. At the end of this section we complete the proof of Theorem~\ref{mainthm-rest} as per this algorithm.

\subsection{The Algorithm}
\begin{enumerate}

\item[(0)]
Let $\Gamma$ be a cofinite subgroup of $\Gamma_n$.
Set up the conjecture for the correct size of the $L^\infty$-size of the Bergman kernel using the lower bound technique from, say, Section~\ref{lbds-all}. it seems that this strategy is better suited in general, as (even asymptotic) dimension formulae for spaces of automorphic forms are not easy to prove or find.

    \item[(1)]
    Pick a conjugate $G$ of $\Gamma$. Consider the geometric expression of $B_{k,G}(Z)$ from \eqref{bergdef0} with $Z \in \fn$.

\item[(2)]
Show that for all $Z \in \fn$ and all $T \in \Lambda_\Gamma$, one has $P_{T; G}(Z) \ll_n 1$.

\item[(3)]
Show that
\begin{align} \label{algo-gam-lat}
   \det(Y)^{k}\, B_{k,G}(Z) \ll  |\mcr S_G|^{-1} \, b_{n,k}^{-1} \sumn_{T \in \Lambda_G} \det(T)^{k-\frac{n+1}{2} } \exp(- 2 \pi \, T Y ) \ll k^{\frac{n(n+1)}{4}} .
\end{align}

\item[(4)]
Repeat the Steps~(1)--(3) for all conjugates of $\Gamma$. Conclude from the relation \eqref{bb-bk}.
\end{enumerate}

We now give some heuristics or indications to illustrate why the Algorithm might work for all congruence subgroups. It already works for $\Gamma=\Gamma_n$ (cf. Theorem~\ref{mainthm-wt}) and $\Gamma^{(n)}(N)$ (cf. Theorem~\ref{mainthm-prin}). Moreover, we see that only 
Steps (2) and (3) need attention. 

\subsubsection{Discussion of (2) in certain cases not considered before}
For this one needs to look more closely at the group $\mcr U(G)$. We first note that for any $m(U) \in \mcr U(\Gamma)$, one has 
\begin{align}
m(U)  \psmb 1_n & \mcr S_G \\ 0 & 1_n \psme m(U)^{-1} \subset \psmb 1_n & \mcr S_G \\ 0 & 1_n \psme  \text{   i.e.   } U  \mcr S_G U^t \subset \mcr S_G;
\end{align}
this places congruence restrictions on the elements $U \in \mcr U_G$.
For simplicity we assume that that $n=2$, $N$ odd and $n_{11}, n_{12}, n_{22}$ are pairwise co-prime. Then one arrives, in particular, at the following conditions:
\begin{align}
    n_{11} | u_{12}, \, n_{22} | u_{21}, \, n_{12} | (u_{11}u_{21}, u_{12}u_{22}) \q \q (U = \psmb u_{11} & u_{12} \\ u_{21} & u_{22} \psme).
\end{align}
We assume that $G \subset \Gamma^{(2)}_0(\omega_G)$ for all $G$.
In view of \lemref{cgnot0-lem}, it is enough to show that for all $T \in \Lambda_G$,
\begin{align}
    H_{G}(T, Y) = \sumn_{U \in \mcr U_\Gamma} \exp(- 2 \pi \tr \, \mc T  Y[U] ) \ll     H_{G}(T, 1_n) \ll 1.
\end{align}
We write $UU^t= \psmb \mc U_1 & \mc U_2 \\ \mc U_2 & \mc U_4 \psme$, and notice that
\begin{align}
    \tr \, T[U] = t_1 \mc U_1 + t_2 \mc U_2 + t_{12} \mc U_{12} \ge (\sqrt{t_1 \mc U_1} - \sqrt{t_2 \mc U_2} )^2 + (\sqrt{t_1t_2} - |t_{12}| ) \sqrt{\mc U_1 \mc U_2}.
\end{align}
Whenever $\sqrt{t_1t_2} - |t_{12}| \gg \omega^{-1/2}$ we can be successful. Notice that this condition is fulfilled if $t_{12}=0$ or if $t_1 t_2$ is large depending on the parameters $n_{ij}$.
Namely, we will get
\begin{align}
       H_{G}(T, 1_n) \ll \sumn_{U} \exp(- \frac{2 \pi}{\sqrt{\omega} } \, \sqrt{\mc U_1 \mc U_2}).
\end{align}
Here $\omega = n_{11} n_{12} n_{22}$. It is clear (cf. the argument preceding \eqref{hgam2}) that we will be done if we can show that
 $ \displaystyle  \sqrt{\mc U_1 \mc U_2} \gg \omega$. But this follows from the inequalities:
 \begin{align}
     \mc U_1 \ge n_2 \sqrt{n_{12}}, \, \mc U_2 \ge n_1 \sqrt{n_{12}}.
 \end{align}
 Finally we can then write (noting that $(\mc U_1 \mc U_2)^{1/2} \ge \sqrt{\omega} (\mc U_1 \mc U_2)^{1/4}$)
 \begin{align}
     H_{G}(T, 1_n) \ll \sum_{U} \exp(- 2 \pi  \, (\mc U_1 \mc U_2)^{1/4}) \ll \sum_{(u_{11}, u_{21})} (u^2_{11} + u^2_{21})^{-3} \cdot \sum_{(u_{12}, u_{22})} (u^2_{12} + u^2_{22})^{-3} \ll 1. \QEDB
 \end{align}

\subsubsection{The Step (3)} \label{step3}
This Step works for all congruence subgroups of $\Gamma_n$, see Lemma~\ref{bkg-lips}.

\section{Ancillary results on the Bergman kernel and applications}

\subsection{Polynomial bound in \texorpdfstring{$\mc F^{(n)}_1$}{}}
Let $\Gamma \subset \Gamma_n$ be any congruence subgroup. Let us define, for $g \in \Gamma$,
\begin{equation}
\mc S_k(g):=\sumn_{S \in \sab} \det ( Z -\overline{g(Z)}  +S)^{-k}. 
\end{equation}
Notice that there is a natural inclusion $\Gamma_\infty \backslash \Gamma \hookrightarrow \Gamma_{n,\infty} \backslash \Gamma_n$. Also note that the positive quantity $|\mc S_k(g)|$
is $G_\infty$-invariant. Thus by looking at \eqref{bergdef1}, we get
by bounding absolutely that,
\begin{align}
     B_{k,\Gamma}(Z) \le  \sumn_{g \in \Gamma_\infty \backslash \Gamma} |J(g,Z)|^{-k} |\mc S_k(g)| \le \sumn_{\gamma \in (\Gamma_{n})_\infty \backslash \Gamma_n} |J(\gamma,Z)|^{-k} |\mc S_k(\gamma)|.
\end{align}
Instead of using the bound \eqref{lip-bd} for $\mc S_k(g)$, we use the  bound from \cite[Lem.~6.3, second bound]{das-krishna}:
\begin{align}
    \mc S_k(g) \le \sumn_{S \in \symn} |\det ( Z -\overline{g(Z)}  +S)|^{-k} \ll_\epsilon k^\epsilon \, \det(Y)^{-k+ \frac{n+1}{4}}.
\end{align}
One immediately obtains (cf. \cite[Thm.~6.5, bound~(2)]{das-krishna}) the following, which we summarize in a proposition. The first bound is from the discussion above, and the second is from \eqref{lip-bd}.

\begin{prop} \label{pol-bd-f1}
For any any congruence subgroup $\Gamma$ and $Z \in \fn$ and $k>(n+1)^2$, one has
\begin{align} \label{bkg-2}
  \bkngam    \ll_n \min\{ k^{\frac{n(n+1)}{2} + \epsilon} \det(Y)^{\frac{3(n+1)}{4}+\epsilon} , \, k^{\frac{3n(n+1)}{4}} \det(Y)^{\frac{(n+1)}{2}+\epsilon}\}. 
\end{align}
\end{prop}

\subsection{Proof of \thmref{mainthm-klarge}} \label{klarge-sec}
Our argument will not exactly follow the algorithm in Subsection~\ref{gen-str}, but will be an application of \propref{pol-bd-f1} and the bound (recall $\omega=\omega_G$)
\begin{align} \label{bkg-3}
\mbb B_{k,G}(Z) \ll k^{\frac{3n(n+1)}{4}}  \, \big( \frac{\omega^n}{\det(Y)} \big)^{\frac{n}{2} + \epsilon} 
\end{align}
from \eqref{bkg-1}. Namely we apply \eqref{bkg-3} in the region $\det(Y) \ge \omega^n$, where thus $\mbb B_{k,G}(Z) \ll k^{\frac{3n(n+1)}{4}} $. Next, in the region $\det(Y) \le \omega^n$, we use \eqref{bkg-2} to get
\begin{align}
    \mbb B_{k,G}(Z) \ll_n k^{\frac{n(n+1)}{2} + \epsilon} \, \omega^{\frac{3n(n+1)}{4}+\epsilon}. 
\end{align}
Since we want this to be $\ll k^{\frac{3n(n+1)}{4}}$, we find that $N^{3+\epsilon} \le k$ suffices, since $\omega|N$. This completes the proof of \thmref{mainthm-klarge}. \QEDB

As a further corollary we can estimate the growth of the Bergman kernel in the entire Siegel's fundamental domain $\fngam$. This result may be useful elsewhere. It also shows the expected bound in $\fngam$ for the Bergman kernel.

\begin{cor} \label{pol-bd-fn}
    For any any congruence subgroup $\Gamma$ and $W=U+iV \in \fngam$ and $k>(n+1)^2$, one has for any $\epsilon >0$,
\begin{align} \label{fn-bd-n}
    \mbb B_{k,\Gamma}(W) \ll_{n,\epsilon} k^{\frac{3n(n+1)}{4}} \, \det(V)^{-\frac{n+1}{2}-\epsilon}.
\end{align}
\end{cor}

\begin{proof}
    We can find coset representatives $\{\gamma\}$ of $\Gamma$ (of level $N$) in $\Gamma_n$ such that for all $\gamma$, one can ensure that both $c_\gamma, c_{\gamma^{-1}}$ are non-singular. This is because given any set of representatives $\{ \tilde{\gamma} \}$, and a prime $p \nmid N$, we can consider the system of congruences:
    \begin{align}
        \gamma \equiv \tilde{\gamma} \bmod N; \q \gamma \equiv \psmb 0 & -1_n \\ 1_n & 0 \psme \bmod p.
    \end{align}
    Then $\gamma^{-1} \equiv \psmb 0 & 1_n \\ -1_n & 0 \psme \bmod p$.
    Clearly $\det(c_\gamma) \neq 0, \det(c_{\gamma^{-1}}) \neq 0$ and $\{\gamma\}$ is still a set of coset representatives, since $\gamma = g \tilde \gamma$ with $g \in \Gamma^{(n)}(N) \subset \Gamma$.

    Now one can finish the proof easily by writing $W=\gamma Z$ for some $\gamma$ as above and $Z \in \fn$. Then one gets (applying the second bound in \propref{pol-bd-f1} to the conjugates of $\Gamma$)
    \begin{align} \label{bkgam-bd}
        \mbb B_{k,\Gamma}(W) = \mbb B_{k, \gamma^{-1} \Gamma \gamma}(Z) \ll_{n,\epsilon} k^{\frac{3n(n+1)}{4}} \, \det(Y)^{\frac{n+1}{2}+\epsilon} .
    \end{align}
Let us write $\gamma^{-1}= \psmb A & B \\ C & D \psme $. The proof now follows from the inequalities 
\[ |\det(CW+D)| = |\det(C)| \det(W+C^{-1}D)| \gg |\det(C)| \det(V)| \ge |\det(V)| \]
since we know that $|\det(C)| \ge 1$. 
Now $Z=\gamma^{-1} W$ which implies that $\det(Y)=\det(V) |\det(CW+D)|^{-2}$.
Therefore from \eqref{bkgam-bd},
\begin{align}
    \mbb B_{k,\Gamma}(W) \ll_n k^{\frac{3n(n+1)}{4}} \, \det(Y)^{\frac{n+1}{2}+\epsilon} \ll  k^{\frac{3n(n+1)}{4}} \left( \frac{\det(V)}{|\det(CW+D)|^2} \right)^{\frac{n+1}{2}+\epsilon} \ll \frac{ k^{\frac{3n(n+1)}{4}} }{ \det(V)^{\frac{n+1}{2}+\epsilon} }.
\end{align}
This finishes the proof of the Corollary.
\end{proof}

\section{The small weights} \label{small-wts}
In the Section we will deal with the so called small weights: $k \ge n/2$. For better exposition we feel that the case $n=1$ should be first discussed.

It is clear that the strategy of the earlier Sections would not work here, as the relevant Poincar\'e series do not converge absolutely. One can however use Hecke's trick e.g. to construct such series $P_{2,\Gamma}$ for the weight $k=2$. One can write down their Fourier expansion (cf. \cite{smart}), but there is some issue about bounds on Kloosterman sums on general congruence subgroups uniform in all parameters, see \hypref{kloo-hypo}. Anyway for higher degrees this is definitely a big challenge. 

Since our main objective here is not to obtain explicit Fourier expansions of Poincar\'e series (cf. earlier Sections, where we did not use the Fourier expansion at all), we approach this issue differently.  We will now propose various remedies to this situation.

We may use holomorphic weight increasing differential operators, and obtain information about the sizes of their Fourier coefficients on small weights from those on bigger weights. The information may be even worse than Hecke's bound, but perhaps this would be sufficient as we need bounds only in the region $\fn$. This kind of argument was also used in \cite{das-krishna}. For instance, one may use the Serre derivative for weights $k \ge 1$:
\begin{align}
    \vartheta_k \colon S_k(\Gamma) \rightarrow S_{k+2}(\Gamma); \q f \mapsto \theta f - \frac{k}{12} f \cdot E_2,
\end{align}
where $\theta f = q \frac{d}{dq} f$ and $E_2$ is the holomorphic quasimodular form of weight $2$. If one knows the explicit image (in wieght $4$) of $P_{2,\Gamma}$ under $\vartheta_2$; one can recursively solve for the coefficients of $P_{2,\Gamma}$ in terms of that of the image. This information may then be used to bound the Poincar\'e series of small weight. Namely for the coefficients $a(n)$ of the small weight Poincar\'e series, one obtains the recursion: $na(n)=\sum_{m<n} u(a(n))$ for some function $u$. Also the $L^2$ norms may be handled via `weight-increasing' or `decreasing' Maa{\ss} operators. See \cite{b-will} for a description of this `image', however here $k=2$ is not considered. For higher degrees one may use suitable Rankin-Cohen operators.

Here we offer the following solution to the question about small weights for any fixed cusp form. The bounds are what one would get at best from the corresponding average bounds. This may be thought of as a `weight-embedding' technique as compared to the `level-embedding' technique of \cite{dfi}.

\begin{hyp} \label{small-hypo}
    We assume that $\sup(S^n_\kappa(\Gamma)) \ll_n \kappa^{2 \alpha}$ ($\kappa \gg_n 1$, $\alpha \ge 0$ depending only on $n$) where the implied constant depends only on the complementary set of parameters if we consider weight or level or the hybrid aspects respectively. We take $\alpha=0$ in the hybrid aspect.
\end{hyp}

\begin{thm} \label{small-wt-thm}
    Let $k \in \z$ satisfy $k \ge \frac{n}{2}$ and $\Gamma \in \Gamma_n$ be any congruence subgroup. With the \hypref{small-hypo}, for any $L^2$-normalised $G \in \skgam$, one has the bound \begin{align}
        \norm{G}_\infty \ll_n \kappa^\alpha.
    \end{align}
    The implied constants in each case depend on the complementary set of parameters.
\end{thm}
Notice that strictly speaking, only the level aspect makes sense for weights $k \ll n$. 

\begin{proof}
    The proof is based on an interpolation method using the various $L^p$ norms. Fix any non-zero $F \in S^n_k(\Gamma)$. We do not normalize by the Petersson norm yet.
    For ease of notation, let us put $f(Z) = \det(Y)^{k/2} F(Z)$.
    Let $m \ge 1$. We first note that
    \begin{align}
        \sup \nolimits_{\hn} |f|^m = (\sup \nolimits_{\hn} |f|)^m.
    \end{align}
    We choose $m \in \mf N$ large enough so that $mk= \kappa $ is bigger than the weights for which Hypothesis~\ref{small-hypo} holds. Note that we can, and will assume that $m$ depends only on $n$, since our previous results are valid for all $\kappa \gg_n 1$. (e.g. \thmref{mainthm-wt}, \thmref{mainthm-prin} are valid for $\kappa \gg n$ etc.)
    Thus using Hypothesis~\ref{small-hypo} for the weight $\kappa$,
    \begin{align}
        |f|^m \ll (km)^\alpha \norm{f^m}_2 \ll_n k^\alpha \left(\int |f|^{2m} \right)^{1/2} \ll_n k^\alpha\norm{f}^{m}_{2m}.
    \end{align}
We now use H\"{o}lder's inequality in the form (see, e.g. \cite{bky})
\begin{align}
    \norm{f}_{2m} \le \norm{f}_2^{1/m} \cdot \norm{f}_\infty^{1-1/m}.
\end{align}
This gives us
\begin{align}
    |f|^m \ll k^\alpha \norm{f}_2 \cdot \norm{f}_\infty^{m-1},
\end{align}
which in turn implies, by taking supnorms,
\begin{align}
    \norm{f}_\infty \ll k^\alpha \norm{f}_2,
\end{align}
    finishing the proof of the theorem, by taking $G=F/\norm{F}_2$.
\end{proof}

\section{Appendix: the case \texorpdfstring{$n=1$}{} } \label{app}

In this Section we show how our ideas in this paper play out successfully for elliptic modular forms. For weights $k>2$, the proof is especially short -- and for the more subtle case $k=2$, we develop a version of large sieve inequality for Fourier coefficients for \textit{arbitrary} congruence subgroups -- something which seems not available in the literature. We do need this result for arbitrary congruence subgroups, since given a congruence subgroup $\Gamma$ of level $N$, we have to simultaneously work with all its conjugates. This will lead to a bound $O(\log q)$ for the size of $S_2(\Gamma)$,  for most $\Gamma$ of level $q$. This is slightly off the expected bound $O(1)$, but is sufficient for analytic applications. We follow \cite{dfi} and \cite{duke} for this purpose.

\subsection{The case \texorpdfstring{$n=1$}{} and \texorpdfstring{$k>2$}{}}
This is the easy case and can be dealt with immediately following the ideas already discussed in this paper.

\begin{thm} \label{n=1,k>2}
    For any congruence subgroup of $\Gamma \subset \Gamma_1$ and $k > 2$, one has $\displaystyle \sup(S_k(\Gamma)) \asymp 1$, with the implied constant absolute. 
\end{thm}

\begin{proof}
The basic pathway of the proof of is essentially contained in the proof of \thmref{mainthm-wt}. We  keep the same setting as that used loc.  cit.
Here there is no sum over $U$, and we simply note that on the piece $\gamma_j \mc F^{(n)}_1$, with $z \in \mc F_1$ and $G=\gamma_j^{-1} \Gamma \gamma_j$, the bound
\begin{align}
    B_{k,G}(z) & \le \omega^{-k} \, b_{1,k}^{-1} \sumn_n n^{k-1 } |P_{n;G}(z)| e(-2 \pi n y /\omega ).
\end{align}
The remaining task is to show that $P_{n;G}(z) \ll 1$ for $z \in \mc F_1$, with an absolute implied constant, which is obvious if we just majorize by the majorant $O(\sum_{c,d} |cz+d|^{-3}) $, which is absolutely bounded since $k \ge 3$ and $y \gg 1$. We now invoke \eqref{lip-bd} to finish the proof.
\end{proof}

\subsection{The case \texorpdfstring{$n=1$}{} and \texorpdfstring{$k=2$}{}}
Let $\Gamma \subset \sltwo$ be a congruence subgroup of level $q$ and let $\omega$ be the width of the cusp $\infty$ of $\Gamma$. The bound $O(\log q)$ for square-free levels $q$ and when $\Gamma=\Gamma_0(q)$ was obtained in \cite{mich-ullmo}, which was improved to $O(1)$ in \cite{jorgenson2004bounding}.
In the large sieve inequality that we want, it turns out that the relevant parameter here is $\omega$ and not $q$. It seems this inequality is not available in the literature, yet this is crucial for us.
With this in mind, we proceed as in \cite[Sec.~5.4]{iwaniec-aut} as follows. We define the normalized Fourier coefficients of $f = \sum_{n \ge 1} a_f(n) e(nz/\omega) \in S_2(\Gamma)$ by
\begin{align}
    a(f,n) = \frac{a_f(n)}{(4 \pi n)^{1/2}}.
\end{align}
For any sequence $a=(a_n)$ of complex numbers, we further define the linear form
\begin{align}
    \mc L(f, a) = \sumn_{n \le K} a_n a(f,n).
\end{align}
We next quote the formula for the Fourier coefficients of (analytically continued) Poincar\'e series of weight $2$ (see \cite{smart}, \cite{parson}). Let us define the Kloosterman sum for $\Gamma$ at the cusp $\infty$ (for $c \ge  1$)by
\begin{align}
    S_\Gamma(m,n;c) = \sumn_{d \in D(c)} e\big( \frac{m \bar{d} + n d}{c \omega} \big)
\end{align}
be the Kloosterman sum at the cusp $\infty$ for $\Gamma$. Here $D(c)=\{ d \in \z \mid (c,d) = (c_\gamma, d_\gamma) \text{ for some } \gamma \in \Gamma; d \in [0, c \omega] \}$.

\begin{prop} \label{k2-prop}
\begin{align} \label{k2-formula}
    p_m(n) = 2\delta(m,n) - \frac{4 \pi}{\omega} (\frac{n}{m})^{1/2} \sum_{c=1}^\infty \frac{S_\Gamma(m,n;c)}{c} J_1 \big( \frac{4 \pi \sqrt{mn}}{c \omega} \big),
\end{align}
where $J_1(x)$ denotes the Bessel function of the first kind.
\end{prop}
We have the Petersson trace formula in this setting:
\begin{equation} \label{pet-tr}
    \sumn_{f \in \mcr B_2(\Gamma)} \overline{a(f,m)} a(f,n) = \frac{2}{\omega^2} \big( \frac{m}{n} \big)^{1/2} \, p_m(n).
\end{equation}

\begin{hyp} \label{kloo-hypo}
We make the following assumption about the general Kloosterman sum:
\begin{align}
    S_\Gamma(m,n;c) \ll (c\omega)^{1/2+\epsilon} (m,n,c\omega)^{1/2}.
    \end{align}
    It is known from Goldfeld-Sarnak (\cite{go-sa}) that $\displaystyle \limsup_{c \to \infty} |S_\Gamma(m,n;c)|c^{-1} \le 1/2 $, and from Petersson (\cite{petersson}) that $S_\Gamma(m,n;c) \ll_\epsilon c^{1/2 +\epsilon}$ with the implied constant not depending on $n$. With these evidences, it seems likely that the above hypothesis may be proved from known techniques.
\end{hyp}
We want to offer two proofs of the desired large sieve inequality -- one via the argument of Duke \cite{duke} and the other via that of Duke-Friedlander-Iwaniec \cite{dfi}. Of these, the former has two advantages over the latter: most importantly, it does not rely on \hypref{kloo-hypo} on Kloosterman sums and secondly it gives a slightly better bound. We nevertheless include both approaches: this not only serves as a cross check for the result, which is a bit different in terms of a scaling factor than that for the Hecke congruence subgroups $\Gamma_0(q)$, but also perhaps provides a different perspective applicable to other situations.

\begin{prop} \label{duke-gam}
   Let $\Gamma \subset \Gamma_0(q)$ be any congruence subgroup of level $q$, and $\omega=\omega_\Gamma$ be the width of the cusp at $\infty$.
   
   (A) Then with the above notation, one has the bound 
    \begin{equation} \label{ls-gamma}
    \sumn_{f \in \mcr B_2(\Gamma)} |\mc L(f, a)|^2 \ll \frac{1}{\omega^2} \big( \log K + O(\frac{K}{q \omega}) \big) \norm{a}^2 .
    \end{equation}

(B)  In addition, assume \hypref{kloo-hypo}. Then,
\begin{equation} \label{ls-gamma2}
    \sumn_{f \in \mcr B_2(\Gamma)} |\mc L(f, a)|^2 \ll \frac{1}{\omega^{2-\epsilon} } \big( 1+ O(\frac{K}{q \omega} ) \big) \norm{a}^2 .
    \end{equation}
\end{prop}

\begin{proof}[Proof of (A)]
By the duality of the large sieve (see e.g. \cite{duke}), it is enough to prove for any $f \in \mcr B_2(\Gamma)$ that
\begin{align} \label{no-norm}
    \sumn_{n \le K} |a_f(n)|^2 \ll \frac{K}{\omega^2}  \big( 1 + O(\frac{K}{q \omega}) \big).
\end{align}
Then \thmref{duke-gam}, which is for the normalised Fourier coefficients $a(f,n) \asymp a_f(n) n^{-1/2}$, follows from \eqref{no-norm} by partial summation.

We will only be brief and mention the important points. We start with
\begin{align}
    \omega \sumn_{n \le K} |a_f(n)|^2 \exp(- 4 \pi n y/\omega) = \int_{0}^\omega |f(x+iy)|^2 dx.
\end{align}
Keeping in mind that $k=2$ here, we arrive at the inequality
\begin{align}
    \omega Y \int_{1}^\infty \exp(- 4 \pi KY y/\omega)  dy \cdot \sumn_{n \le K} |a_f(n)|^2 \le \int_Y^{\infty} \int_{0}^\omega |f(z)|^2  dxdy.
\end{align}
It we let $P_{Y}:= \{ z =x+iy \in \h \mid 0<x\le \omega,\,  y>Y\}$, then as in \cite{duke}, we can write that
\begin{align}
    \int_{P_Y} |f(z)|^2 dxdy \le \max_{z \in P_Y} \#\{ \gamma \in \Gamma \mid \gamma(z) \in P_Y\} \pet{f}
\end{align}
because $P_Y$ can be covered by finitely many translates of the fundamental domain $\mc F_\Gamma$, which we assume to be without loss inside the strip $\{ 0<\re(z) \le \omega \}$ (recall that one has $\Gamma_\infty=< \psmb 1 & \omega \z \\ 0 & 1 \psme > $). The count of $\gamma$ goes through without any change as only imaginary parts of $z, \gamma(z)$ play a role. We get
\begin{align}
    \max_{z \in P_Y} \#\{ \gamma \in \Gamma \mid \gamma(z) \in P_Y\} \ll 1 + 3 q^{-1} Y^{-1}.
\end{align}
At this point, choosing $Y:= \omega K^{-1}$ finishes the proof of (A).
\end{proof}

\begin{proof}[Proof of (B)]
    We will closely follow the argument given in \cite[Sec.~5.4]{iwaniec-aut}, and mention only the relevant points. Opening the square on the LHS of \eqref{ls-gamma2} and using \eqref{k2-formula}, \eqref{pet-tr}, we get
\begin{align} \label{lfa1}
    \sum_{f \in \mcr B_2(\Gamma)} |\mc L(f, a)|^2 = \frac{2}{\omega^2} \bigg( 2 \norm{a}^2 + \sum_{c \ge 1} \frac{1}{c \omega} {\underset{m,n \le K}{\sum \sum} } \overline{a_m} a_n \, S_\Gamma(m,n;c) \,  J_1 \big( \frac{4 \pi \sqrt{mn}}{c \omega}   \big) \bigg )
\end{align}
Analogously as in \cite[Sec.~5.4]{iwaniec-aut} we arrive at
\begin{align} \label{smn1}
   | {\underset{m,n \le K}{\sum \sum} } \overline{a_m} a_n \, S_\Gamma(m,n;c)| \le c \omega {\underset{m \equiv n \bmod{c \omega}}{\sum \sum} } \overline{a_m} a_n \le (c \omega +K) \norm{a}^2.
\end{align}
Here we have used the fact that the sets $\{d \in D(c,\Gamma) \mid d \bmod{c \omega}\}$ and $\{a \in D(c,\Gamma) \mid a \bmod{c \omega}\}$ are in bijection via $\gamma \mapsto \gamma^{-1}$. 
Assuming the bound from \hypref{kloo-hypo} and arguing as in [DFI], we also get the bound (this bound is relevant only for $k=2$)
\begin{align} \label{smn2}
    | {\underset{m,n \le K}{\sum \sum} } \overline{a_m} a_n \, S_\Gamma(m,n;c)| \le (c\omega)^{1/2+\epsilon} K \norm{a}^2.
\end{align}
When $K \le \omega$, we use the Taylor series for the Bessel function and \eqref{smn1} to get that the summand in \eqref{lfa1} as a function of $c$ is
 $ \displaystyle  O \big(  \frac{K}{c^2\omega^2} (c \omega + K) \norm{a}^2$ \big),
which we use to bound the contribution of $c \le K$.
Whereas if we use \eqref{smn2}, the same quantity as above is
$ \displaystyle  O \big(  \frac{K}{c^2\omega^2} (c\omega)^{1/2+\epsilon} K \norm{a}^2 \big)  $,
which we use to bound the contribution of $c > K$.
The rest of the argument proceeds as in \cite[Sec.~5.4]{iwaniec-aut} and as long as $q \omega \ge K$ and the bound reads $\displaystyle O(\omega^{-2 +\epsilon})$. 

When $q \omega  <K$, one resorts to a level embedding technique, which is a bit subtle here. For a prime $p$, we take $\Gamma':= \Gamma^0(p) \cap \Gamma$, and note that $[\Gamma:\Gamma'] \le p+1$ from the natural inclusion $\Gamma'\backslash \Gamma \hookrightarrow \Gamma^0(p) \backslash \Gamma_1$. We further assume $(p,\omega)=1$, so that the width of $\Gamma'$ at $\infty$ is $p\omega$. We choose $p$ such that $K < p q \omega \ll K$ and one note that 
\begin{align}
\displaystyle \sum_{f \in \mcr B_2(\Gamma)} |\mc L(f, a)|^2 \le (p+1) \sum_{f \in \mcr B_2(\Gamma')} |\mc L(f, a)|^2 \ll \frac{1}{p^2\omega^{2-\epsilon} } (p+1) \ll \frac{1}{\omega^{2-\epsilon} } \big( 1 + O(\frac{K}{q \omega}) \big)  \norm{a}^2 ,
\end{align}
as desired. This finishes the proof.
\end{proof}

\begin{thm}
Let $\Gamma \subset \Gamma_0(q)$ be any congruence subgroup of level $q$. Then   one has the bound $\sup(S_2(\Gamma)) \ll \log q$.
\end{thm}

\begin{proof}
    The proof is immediate from the large sieve inequality \propref{duke-gam}~(A). In view of the Algorithm~\ref{algo-gam-lat}, is enough to prove that $\mbb B_G(z) \ll 1$ for all conjugates $G $ of $\Gamma$ and for all $z \in \mc F_1$.
    
    First of all, using the trivial bound for the Kloosterman sums $S_\Gamma(m,n;c)$ one can easily prove a bound 
    $\displaystyle a(f,n) \ll n^{D}$ for some absolute $D>0$. We can then truncate the Fourier expansion at $K= \delta q \log q$ for appropriate $\delta$ depending only on $D$, for the normalized Fourier coefficients $a(f,n)$ for any $f \in \mcr B_2(\Gamma)$, via the Petersson formula. E.g., see \cite{mich-ullmo} where arithmetic new-old basis decomposition and Deligne's bound were used. We don't need any of them. The proof is finished by noting that
    \begin{align}
     \mbb B_G(z) \le  \sumn_{f \in \mcr B_2(\Gamma)} y^2\left| \left(\sumn_{n \le K} + \sumn_{n > K} \right) a_f(n) e(nz/\omega) \right|^2.
    \end{align}
We define the sequence $a=(a_n)$ by $a_n:= \sqrt{n}  e(nz/\omega)$ and put $\omega=\omega_G$. Then we can write,
\begin{align}
    \mbb B_G(z) \le \sumn_{f \in \mcr B_2(\Gamma)} y^2|\mc L(f, a)|^2 + \sumn_{f \in \mcr B_2(\Gamma)} y^2 |\sumn_{n > K} a_f(n) e(nz/\omega) |^2
\end{align}
Let us call the first and second sums above as I and II respectively. Then 
\begin{align} \label{I}
    \mrm{I} \ll \frac{y^2}{\omega^2}(1+ \frac{K}{q \omega}) \sum_{n \le K} n \exp(- 4 \pi n y/\omega) \ll (1+ \frac{K}{q \omega}) \ll \log q.
\end{align}
A bound for the $n$-sum is precisely worked out on \cite[p.~662]{mich-ullmo} with $y$ replaced with $y / \omega$, $N$ with $\omega$. This works mutatis mutandis since $\omega|N$. We require $\delta>1/4 \pi$. We do not repeat the details.

For the sum II, one actually has an exponential decay as it can be compared with the incomplete Gamma function. Note that for us $y \gg 1$. 
\begin{align}
    y  \sumn_{n > K}  a_f(n) e(nz/\omega)  \le y \sumn_{n > K}  a_f(n) \exp(- 2 \pi  ny/\omega) 
     \ll y \sumn_{n > K}  n^{2D} \exp(- 2 \pi  n y/\omega) .
\end{align}
Recall (see e.g. \cite{pinnelis}) that the incomplete Gamma function, for $a>1$, defined by $\displaystyle \Gamma(a,x):= \int_x^\infty t^{a-1} e^{-x} dx$ satisfies the bound
\begin{align}
    \Gamma(a,x) \le x^{a-1}e^{-x} \left(1- \frac{a-1}{x} \right)^{-1}, \q x>a-1.
\end{align}
Recall that $K= \delta q \log q$ for suitable $\delta>0$. We thus have 
\begin{align}
    y  \sumn_{n > K}  a_f(n) e(nz/\omega) & \ll y \cdot \left(\frac{\omega}{y} \right)^{D+1} \cdot \Gamma \left(D+1, \frac{K}{\omega} \right) \\
    &\ll q^{D+1} K^D \exp\left(- \frac{K}{\omega} \right) \ll q^{3D} \exp \left(- \frac{K}{q} \right) \ll q^{3D} \exp (- \delta \log q) \ll 1 \label{II}
\end{align}
for $\delta \gg D$. Since $D$ was an absolute constant, this finishes the proof if we combine the bounds in \eqref{I} and \eqref{II}.
\end{proof}

\relax

\printbibliography

\end{document}